\documentclass{amsart}
\usepackage{amssymb}
\usepackage{geometry}
\usepackage{amssymb,amsbsy}
\usepackage{epstopdf}
\usepackage{float}
\usepackage{amsfonts}
\usepackage{amsmath}
\usepackage{amscd}
\usepackage{tikz-cd}
\usepackage{latexsym}
\usepackage{bbm}
\usepackage{verbatim}
\usepackage{paralist}
\usepackage[utf8]{inputenc}
\usepackage{etoolbox}
\usepackage{enumitem}
\usepackage{xspace}
\usetikzlibrary{patterns}
\newtheorem{theorem}{Theorem}[section]
\newtheorem{corollary}[theorem]{Corollary}
\newtheorem{lemma}[theorem]{Lemma}
\newtheorem{observation}[theorem]{Observation}
\newtheorem{proposition}[theorem]{Proposition}

\newtheorem{fact}[theorem]{Fact}
\newtheorem*{thm*}{Theorem} 

\theoremstyle{definition}
\newtheorem{definition}[theorem]{Definition}
\newtheorem{question}[theorem]{Question}

\theoremstyle{remark}
\newtheorem{remark}[theorem]{Remark}

\renewcommand{\bar}{\overline}

\renewcommand{\P}{\mathbb{P}}
\newcommand{\Q}{\mathbb{Q}}
\newcommand{\Namba}{\mathbb{N}}

\newcommand{\N}{{\overline{N}}}
\newcommand{\G}{\overline{G}}

\newcommand{\ZFC}{\textup{\ensuremath{\textsf{ZFC}}}}
\newcommand{\ZFCm}{\textup{\ensuremath{\textsf{ZFC}^-}}\xspace}
\newcommand{\SCFA}{\textup{\ensuremath{\textsf{SCFA}}}}
\newcommand{\CH}{\textup{\textsf{CH}}\xspace}

\DeclareMathOperator{\cof}{cof}
\DeclareMathOperator{\height}{height}
\DeclareMathOperator{\ran}{range}

\DeclareMathOperator{\Succ}{succ}

\DeclareMathOperator{\Add}{\mathcal A\textit{dd}\,}
\DeclareMathOperator{\Coll}{\mathcal C\textit{oll}\,}

\newcommand{\id}{\textup{\ensuremath{\text{id}}}}
\newcommand{\st}{\; | \;}
\newcommand{\set}[2]{\left\{#1\st #2 \right\}}
\newcommand{\seq}[2]{\langle #1 \st #2 \rangle}
\newcommand{\forces}{\Vdash}
\newcommand{\rest}{\mathbin{\upharpoonright}}

\newcommand{\ubp}{\textup{\textrm{UBP}}}
\newcommand{\sob}{\textrm{Suslin off the generic branch}}

\newcommand{\To}{\longrightarrow}
\newcommand{\kla}[1]{\langle #1 \rangle}
\newcommand{\sub}{\subseteq}
\newcommand{\bN}{{\bar{N}}}
\newcommand{\V}{\mathrm{V}}
\newcommand{\power}{\mathcal{P}}
\newcommand{\bA}{\bar{A}}
\newcommand{\goedel}[1]{{\prec}#1{\succ}}

\DeclareMathOperator{\MPc}{\textup{\textsf{MP}}_{<\omega_1\textup{-closed}}}

\newcommand{\BSCFA}{\ensuremath{\mathsf{BSCFA}}}

\newcommand{\BFA}{\ensuremath{\mathsf{BFA}}}
\newcommand{\MA}{\ensuremath{\mathsf{MA}}}

\newcommand{\ColNothing}{\mathrm{Col}}
\newcommand{\Col}[1]{\ColNothing(#1)}

\begin{document}

\title{Subcomplete Forcing, trees and generic absoluteness}

\author[Fuchs]{Gunter Fuchs}
\address[G.~Fuchs]{Mathematics,
          The Graduate Center of The City University of New York,
          365 Fifth Avenue, New York, NY 10016
          \&
          Mathematics,
          College of Staten Island of CUNY,
          Staten Island, NY 10314}
\email{Gunter.Fuchs@csi.cuny.edu}
\urladdr{http://www.math.csi.cuny/edu/$\sim$fuchs}

\author[Minden]{Kaethe Minden}
 \address[K.~Minden]{Mathematics, Marlboro College, 2582 South Road, Marlboro, VT 05344}
\email{kminden@marlboro.edu}
\urladdr{https://kaetheminden.wordpress.com/}

\thanks{The research of the first author has been supported in part by PSC CUNY research grant 60630-00 48.}

\date{}     					

\begin{abstract}
We investigate properties of trees of height $\omega_1$ and their preservation under subcomplete forcing. We show that subcomplete forcing cannot add a new branch to an $\omega_1$-tree. We introduce fragments of subcompleteness which are preserved by subcomplete forcing, and use these in order to show that certain strong forms of rigidity of Suslin trees are preserved by subcomplete forcings. Finally, we explore under what circumstances subcomplete forcing preserves Aronszajn trees of height and width $\omega_1$. We show that this is the case if \CH{} fails, and if \CH{} holds, then this is the case iff the bounded subcomplete forcing axiom holds. Finally, we explore the relationships between bounded forcing axioms, preservation of Aronszajn trees of height and width $\omega_1$ and generic absoluteness of $\Sigma^1_1$-statements over first order structures of size $\omega_1$, also for other canonical classes of forcing.
\end{abstract}
\maketitle

\section{Introduction}
Much of the work in this paper is motivated by prior work of the first author in which it was observed that the countably closed maximality principle ($\MPc(H_{\omega_2})$) implies countably closed-generic $\mathbf{\Sigma}_2^1(H_{\omega_1})$-absoluteness, defined in Section \ref{sec:GenAbsWideAronszajn}, see \cite{Fuchs:MPclosed}. The point here is that countably closed-generic $\mathbf{\Sigma}^1_1(\omega_1)$-absoluteness is provable in \ZFC. In \cite{Kaethe:Diss}, the maximality principle for subcomplete forcing was considered, and the question arose whether it has the same consequence. Dually to the situation with countably closed forcing, the underlying question is whether subcomplete generic $\mathbf{\Sigma}^1_1(\omega_1)$-absoluteness is provable in \ZFC. Subcomplete forcing was introduced by Jensen, who showed that it cannot add real numbers, yet may change cofinalities to be countable, and that it can be iterated with revised countable support. Moreover all countably closed forcing notions are subcomplete. What makes forcing principles for subcomplete forcing particularly intriguing is that they tend to be compatible with \CH, while otherwise having consequences similar to those of other, more familiar forcing classes, such as proper, semi-proper, or stationary set preserving forcings, which imply the failure of \CH.
There is a close relationship between these generic absoluteness properties and the preservation of certain types of Aronszajn trees, and this led us to investigate properties of trees of height $\omega_1$ and their preservation under subcomplete forcing. The main question we had in mind, stated in \cite[Question 3.1.6]{Kaethe:Diss}, was whether subcomplete forcing can add a branch to an $(\omega_1,{\le}\omega_1)$-Aronszajn tree, that is, a tree of height and width $\omega_1$ that does not have a cofinal branch.

The work on properties of $\omega_1$-trees and their preservation under subcomplete forcing, in particular on strong forms of rigidity, led us to consider weak forms of subcompleteness which themselves are preserved by subcomplete forcing. In Section \ref{sec:FragmentsOfSubcompleteness}, we recall the definition of subcompleteness, originally introduced by Jensen, investigate the relevant fragments of subcompleteness we call minimal subcompleteness and prove the preservation facts we need. In Section \ref{sec:SCandPropertiesOfOmega1Trees}, we show that Suslin trees are preserved under minimally subcomplete forcing, and that such forcing cannot add new branches to $\omega_1$-trees. We then show that certain strong forms of rigidity of Suslin trees, introduced in \cite{FuchsHamkins:DegreesOfRigidity}, are preserved by subcomplete forcing. Finally, in Section \ref{sec:GenAbsWideAronszajn}, we establish the relationships between the preservation of wide Aronszajn trees, forms of generic absoluteness and the bounded subcomplete forcing axiom, of course building on Bagaria's work \cite{Bagaria:BFAasPrinciplesOfGenAbsoluteness} on bounded forcing axioms and principles of generic absoluteness. The main results in this section are as follows. The first one is Theorem \ref{thm:UnderCHEverythingIsClear}, which says:
\begin{thm*}
Assuming \CH, the following are equivalent.
\begin{enumerate}[label=(\arabic*)]
\item Subcomplete generic $\mathbf{\Sigma}^1_1(\omega_1)$-absoluteness.
\item \BSCFA.
\item Subcomplete forcing preserves $(\omega_1,{\le}\omega_1)$-Aronszajn trees.
\end{enumerate}
\end{thm*}
Here, \BSCFA{} is the bounded subcomplete forcing axiom. The corresponding equivalence holds for any other natural class of forcing notions (see Definition \ref{def:StrongAbsoluteness}) that don't add reals.
The second main result is Theorem \ref{theorem:CompleteAnswer}, which settles our original question, whether subcomplete forcing can add a branch to an $(\omega_1,{\le}\omega_1)$-Aronszajn tree:
\begin{thm*}
Splitting in two cases, we have:
\begin{enumerate}[label=(\arabic*)]
  \item If \CH{} fails, then subcomplete forcing preserves $(\omega_1,{\le}\omega_1)$-Aronszajn trees.
  \item If \CH{} holds, then subcomplete forcing preserves $(\omega_1,{\le}\omega_1)$-Aronszajn trees iff \BSCFA{} holds.
\end{enumerate}
\end{thm*}

We have a complete analysis for other forcing classes as well. The following is Theorem \ref{thm:GeneralSituation}.

\begin{thm*}
\label{thm:GeneralSituation}
Let $\Gamma$ be the class of proper, semi-proper, stationary set preserving, ccc or subcomplete forcing notions. Consider the following properties. \begin{enumerate}[label=(\arabic*)]
\item $\BFA_\Gamma.$
\item $\Gamma$-generic $\mathbf{\Sigma}^1_1(\omega_1)$-absoluteness.
\item Forcings in $\Gamma$ preserve $(\omega_1,{\le}\omega_1)$-Aronszajn trees.
\end{enumerate}
Then (1)$\iff$(2)$\implies$(3), but (3) does not imply (1)/(2).
\end{thm*}

Of course, subcomplete forcing is the only one of these forcing classes whose bounded forcing axiom is consistent with \CH, and it is in the presence of \CH that the unusual situation arises that these conditions are equivalent for this class.




\section{Fragments of subcompleteness and their preservation}
\label{sec:FragmentsOfSubcompleteness}

We begin by recalling the concept of subcompleteness of a partial order, as introduced by Jensen (see \cite{Jensen2014:SubcompleteAndLForcingSingapore}). If $M$ and $N$ are models of the same first order language, then we write $M\prec N$ to express that $M$ is an elementary submodel of $N$, and we write $\sigma:M\prec N$ to say that $\sigma$ is an elementary embedding from $M$ to $N$. If $X$ is a subset of the domain of $N$, then we write $X\prec N$ to express that the reduct $N|X$ of $N$ to $X$ is an elementary submodel of $N$.

\begin{definition}
A transitive set $N$ (usually a model of $\ZFC^-$) is \emph{full} if there is an ordinal $\gamma$ such that $L_\gamma(N)\models\ZFC^-$ and $N$ is regular in $L_\gamma(N)$, meaning that if $x\in N$, $f\in L_\gamma(N)$ and $f:x\To N$, then $\ran(f)\in N$.
\end{definition}

\begin{definition}
For a poset $\P$, $\delta(\P)$ is the minimal cardinality of a dense subset of $\P$.
\end{definition}

\begin{definition}
Let $N=L^A_\tau=\kla{L_\tau[A],\in,A\cap L_\tau[A]}$ be a $\ZFC^-$ model where $\tau$ is a cardinal and $A$ is a set, $\delta$ an ordinal and $X\cup\{\delta\}\sub N$. Then $C^N_\delta(X)$ is the smallest $Y\prec N$ such that $X\cup\delta\sub Y$.
\end{definition}

\begin{definition}
\label{def:subcompleteness}
A forcing $\P$ is \emph{subcomplete} if there is a cardinal $\theta>\delta=\delta(\P)$ which \emph{verifies the subcompleteness of $\P$}, which means that $\P\in H_\theta$, and for any $\ZFC^-$ model $N=L_\tau^A$ with $\theta<\tau$ and $H_\theta\sub N$, any $\sigma:\bN\prec N$ such that $\bN$ is countable, transitive and full and such that $\P,\theta\in\ran(\sigma)$, any $\bar{G}\sub\bar{\P}$ which is $\bar{\P}$-generic over $\bN$, and any $s\in\ran(\sigma)$, the following holds. Letting $\sigma(\bar{s},\bar{\theta},\bar{\P})=s,\theta,\P$, there is a condition $p\in\P$ such that whenever $G\sub\P$ is $\P$-generic over $\V$ with $p\in G$, there is in $\V[G]$ a $\sigma'$ such that
\begin{enumerate}[label=(\arabic*)]
  \item $\sigma':\bN\prec N$,
  \item $\sigma'(\bar{s},\bar{\theta},\bar{\P})=s,\theta,\P$,
  \item $(\sigma')``\bar{G}\sub G$,
  \item $C^N_{\delta}(\ran(\sigma'))=C^N_{\delta}(\ran(\sigma))$.
\end{enumerate}
\end{definition}

The three main properties of subcomplete forcings are that they don't add reals, preserve stationary subsets of $\omega_1$, and that they can be iterated (with revised countable support).
We now isolate key parts of what it means that a forcing is subcomplete, which are in a sense responsible for these preservation properties. The remaining parts are crucial for the iterability of the resulting class of forcings. We call the stripped down version of the definition of subcompleteness \emph{minimal subcompleteness}.

\begin{definition}
\label{def:MinimalSC}
If $X$ is a set such that the restriction of $\in$ to $X$ orders $X$ extensionally, then let $\sigma_X:N_X\To X$ be the inverse of the Mostowski collapse of $X$, where $N_X$ is transitive.

Let $N$ be a transitive model of $\ZFC^-$, and let $\P\in X\prec N$, where $X$ is countable. Then $X$ \emph{elevates to $N^\P$} if the following holds: let $c\in N_X$, and let $\bar{\P}=\sigma_X^{-1}(\P)$. Then, whenever $\bar{G}$ is generic over $N_X$ for $\bar{\P}$, there is a condition $p\in\P$ such that if $G$ is $\P$-generic over $\V$ and $p\in G$, then in $\V[G]$, there is an elementary embedding $\sigma':N_X\prec N$ such that $\sigma'``\bar{G}\sub G$ 
and $\sigma'(c)=\sigma_X(c)$.

A forcing notion $\P$ is \emph{minimally subcomplete} if for all sets $a$, $H$, there is a transitive model of $\ZFC^-$ the form $N = L_\tau^A$ with $a\in N$ and $H\sub N$ such that
\[Z_{N,\P,a}=\{ \omega_1 \cap X \st a\in X\ \text{and}\ X\ \text{elevates to}\ \ N^\P\}\]
contains a club subset of $\omega_1$.
\end{definition}

Let us first show that this is indeed a weakening of subcompleteness.

\begin{observation}
\label{obs:MinimalIsWeaker}
If $\P$ is subcomplete, then $\P$ is minimally subcomplete.
\end{observation}

\begin{proof}
Let $\theta$ verify the subcompleteness of $\P$. Given sets $a$, $H$ let $\tau>\theta$, $A\sub\tau$, and $H_\theta\in N=L_\tau[A]$ with $a\in N$ and $H\sub N$. Let $\mu=\tau^+$ and $\nu=\tau^{++}$. Let $N'=L_\mu[A]$. We claim that $Z_{N',\P,a}$ contains a club. Let
\[Z=\{ \omega_1 \cap Y \st a\in Y\prec L_\nu[A],\ Y\ \text{countable}\}\]
Then clearly, $Z$ contains a club. Moreover, if $\omega_1 \cap Y\in Z$, where
$Y\prec L_\nu[A]$ and $Y$ is countable, then, letting $\bA=(\sigma_Y)^{-1}(A)$, $\bar{\tau}=(\sigma_Y)^{-1}(\tau)$ and $\bar{\mu}=(\sigma_Y)^{-1}(\mu)$, it follows that $N_Y$ is of the form $L_{\bar{\nu}}[\bar{A}]$, where $\bA\sub\bar{\tau}$, so $\bA$ is a bounded subset of $\bar{\mu}$, which is regular, and also the largest cardinal in $N_Y$. Letting $\bN=L_{\bar{\mu}}[\bar{A}]$, it follows that $\bN$ is full, as witnessed by $N_Y$ (note that $N_Y$ has the same bounded subsets of $\bar{\mu}$ as $\bN$, and it is a model of $\ZFC^-$). Let $\sigma=\sigma_Y\rest\bN$, $X=\ran(\sigma)=Y\cap L_\mu[A]$. Then, since $\P$ is subcomplete, $X$ elevates to ${N'}^\P$. Since $\omega_1\cap X=\omega_1\cap Y$, it follows that $Z_{N',\P,a}$ contains a club, as claimed.
\end{proof}

Next, let's check that minimal subcompleteness, while weaker than subcompleteness, is still strong enough to preserve the properties mentioned before.

\begin{fact}
\label{fact:MinimalSCenoughForPreservationProperties}
Let $\P$ be a minimally subcomplete forcing, and let $G\sub\P$ be $\P$-generic over $\V$.
\begin{enumerate}[label=(\arabic*)]
\item
\label{item:NoNewReals}
$\power(\omega)^\V=\power(\omega)^{\V[G]}$.
\item
\label{item:SSP}
If $S$ is stationary in $\omega_1$, then this remains true in $\V[G]$.
\end{enumerate}
\end{fact}

\begin{proof}
For \ref{item:NoNewReals}., the proof is exactly the same as with subcomplete forcing.
Assume the contrary, and suppose toward a contradiction that there is name $\dot r \in V^\P$ for a subset of $\omega$  and a condition $q$ forcing that $\dot r$ is new. Let $a = \kla{q, \dot r}$. By minimal subcompleteness, there is $N$ with $H_\theta \subseteq N$ for some large enough $\theta$, so that $Z_{N, \P, a}$ contains a club. So there is $\alpha = \omega_1 \cap X \in Z_{N, \P, a}$, where $X\prec N$ is countable. Then $X$ elevates to $N^\P$. Let $\G$ be a generic filter for $\sigma_X^{-1}(\P)=\overline \P$ over $N_X$, such that $\sigma_X^{-1}(q)=\overline q \in \G$. So we have a condition $p \in \P$, where, letting $G$ be $\P$-generic over $V$ containing $p$, we have an elementary embedding $\sigma': N_X \To N$ such that $\sigma'``\G \subseteq G$. Thus $q \in G$ as well, so we have that $r = \dot r^G$ is new. But this is a contradiction as $r = \sigma'``\overline r = \sigma_X``\overline r=\bar{r}\in V$, where $\overline r=\sigma_X^{-1}(r)$.

For \ref{item:SSP}., assume the contrary. Let $S \subseteq \omega_1$ stationary, and suppose towards a contradiction that there is a $\dot C \in V^{\P}$ such that for some $q\in \P$,
\[q\Vdash ``\dot C \subseteq \check \omega_1 \text{ is club} \ \land \ \check S \cap \dot C = \emptyset."\]
Let $a=\kla{q,\dot{C},S}$, and, by minimal subcompleteness, let
$N$ be such that $Z_{N,\P,a}$ contains a club, where $H_\theta\sub N$, for some $\theta$ which is sufficiently large to conclude that the fact displayed above holds in $N$. Let $\alpha\in S\cap Z_{N,\P,a}$, and let $X$ witness this. That is, $a\in X\prec N$, $\alpha=\omega_1\cap X$, and $X$ elevates to $N^\P$. 
Let $\bar{a}=\sigma_X^{-1}(a)$, $\bar{\P}=\sigma_X^{-1}(\P)$, and let $c=\kla{\bar{a},\bar{\P}}$.

Since $\sigma_X:N_X\prec N$ is elementary, $\bar{a}$ is of the form $\kla{\bar{q},\dot{\bar{C}},\bar{S}}$, where $\sigma_X(\bar{q})=q$, $\sigma_X(\dot{\bar{C}})=\dot{C}$, $\sigma_X(\bar{S})=S$, and in $N_X$, $\bar{q}$ forces with respect to $\bar{\P}$ that $\dot{\bar{C}}\subseteq\check{\alpha}$ is club.

Let $\bar{G}$ be $\bar{\P}$-generic over $\bar{N}$ with $\bar{q}\in\bar{G}$.
Since $X$ elevates to $N^\P$, there is a condition $p\in\P$ such that if we let $G$ be $\P$-generic over $\V$ with $p\in G$, then there is in $\V[G]$ an elementary embedding $\sigma':\bN\prec N$ with $\sigma'(c)=\sigma_X(c)$ and $\sigma'``\bar{G}\sub G$. So $\sigma'$ lifts to an embedding $\sigma^*:N_X[\bar{G}] \prec N[G]$ in $V[G]$.
Let $\bar{C} = \dot{\bar{C}}^{\G}$, $C = \dot C^G$.
Since $q=\sigma'(\bar{q})\in G$, we have that $C \subseteq \omega_1$ is club in $\V[G]$ and $S \cap C = \emptyset$ in $N[G]$. However, $\alpha=\omega_1^\bN$, so $\bar{q}\in\bar{G}$ implies that $\bar{C}$ is club in $\alpha$. Since $\sigma^*\rest\alpha=\id$, it follows that $C\cap\alpha=\bar{C}\cap\alpha$, and so, $\alpha<\omega_1^\V$ is a limit point of $C$, so $\alpha\in C\cap S$, a contradiction.
\end{proof}


It was shown in \cite{Kaethe:Diss} that the subcompleteness of a forcing $\P$ is very fragile: it can be destroyed by countably closed forcing - note in this context that Jensen pointed out that every countably closed forcing is subcomplete. The same negative result remains true of minimal subcompleteness. In the following proposition, $\Namba$ denotes Namba forcing, which Jensen proved to be subcomplete, assuming \CH{} (see \cite{Jensen2014:SubcompleteAndLForcingSingapore} and \cite{Jensen:SCofNambaTypeForcings}).

\begin{proposition}[\cite{Kaethe:Diss}]
Forcing with $\Coll(\omega_1, \omega_2) \times \Namba$ collapses $\omega_1$.
\end{proposition}

Thus, after forcing with $\Coll(\omega_1,\omega_2)$, the ground model version of Namba forcing collapses $\omega_1$, hence adds a real, and is thus not even minimally subcomplete any longer (see Fact \ref{fact:MinimalSCenoughForPreservationProperties}).
However, the minimal fragment of subcompleteness survives countably distributive forcing of size at most $\omega_1$, as we shall show presently.

\begin{lemma}
\label{lemma:MinimalPartOfSubcompletenessIsPreserved}
Let $\P$ be subcomplete. Then after countably distributive forcing of size at most $\omega_1$, $\P$ is minimally subcomplete.
\end{lemma}

\begin{proof}
Let $\Q$ be countably distributive, and let $|\Q| \leq \omega_1$. Let $H \subseteq \Q$ be generic. Let's assume that the conditions in $\Q$ are countable ordinals, so that $H\sub\omega_1$.
To show that $\P$ is minimally subcomplete in $V[H]$, let $a=\dot{a}^H$, $\theta$ be given. In $\V[H]$, we have to find a transitive $N\models\ZFC^-$ with $H_\theta\sub N$, such that the set $Z_{N,\P,a}$ contains a club subset of $\omega_1$.

In $\V$, since $\P$ is subcomplete, we can pick a regular cardinal $\tau$ and an $A\sub\tau$ such that, letting $\mu=\tau^+$ and $\nu=\tau^{++}$, we have that $H_\theta\sub L_{\tau}[A]$, and whenever $Y$ is countable, $\P\in Y$ and $Y\prec L_\nu[A]$, it follows that $X=Y\cap L_{\mu}[A]$ elevates to $L_{\mu}[A]^\P$, because in this situation, $N_X$ is full, as in the proof of Observation \ref{obs:MinimalIsWeaker}.

In $\V[H]$, let $A'=A\oplus H=\{\goedel{\alpha,\beta}\st\alpha\in A\ \land\ \beta\in H\}$, where $\goedel{\alpha,\beta}$ is the G\"{o}del code of $\kla{\alpha,\beta}$. Note that $L_\tau[A']=L_\tau[A,H]=L_\tau[A][H]$. We claim that in $\V[H]$, whenever $Y'\prec L_\nu[A']$ is countable, with $\P,A,A',\Q,H\in Y'$, it follows that $X'=Y'\cap L_\mu[A']$ elevates to ${L_\mu[A']}^\P$.

To see this, fix such $Y'$ and $X'$, and
let $\bar{\tau},\bar{\mu},\bar{A},\bar{A}',\bar{\Q},\bar{H}$ be such that \[\sigma_{Y'}(\bar{\tau},\bar{\mu},\bar{A},\bar{A}',\bar{\Q},\bar{H})=\tau,\mu,A,A',\Q,H\] Then $N_{Y'}=L_{\bar{\nu}}[\bar{A}']$, for some cardinal $\bar{\nu}$, and by elementarity of $\sigma_{Y'}$, it follows that in $N_{Y'}$ it is the case that $\bar{H}$ is $\bar{\Q}$-generic over $L_{\bar{\mu}}[\bar{A}]$. Let $\sigma=\sigma_{Y'}\rest L_{\bar{\mu}}[\bar{A}]$. So $\sigma=\sigma_{X}$, where $X=X'\cap L_\mu[A]$, and $\sigma: L_{\bar{\mu}}[\bar{A}]\prec L_\mu[A]$ is elementary.

In $\V[H]$, let $\bar{G}$ be $\bar{\P}$-generic over $L_{\bar{\mu}}[\bar{A}']$, and let $c\in N_{X'}$ be given. There is then a $\dot{c}\in L_{\bar{\mu}}[\bar{A}]$ such that $c=\dot{c}^{\bar{H}}$. Let $Y=Y'\cap L_\nu[A]$, $X=X'\cap L_\mu[A]$. Since $\Q$ is countably distributive, it follows that $\bar{G},X,Y\in\V$, and $\bar{G}$ is $\bar{\P}$-generic over $L_{\bar{\mu}}[\bar{A}]$. Since $X=Y\cap L_{\mu}[A]$, it follows that $X$ elevates (in $\V$) to $L_\tau[A]^\P$, as $N_X$ is full. Hence, there is a condition $p\in\P$ that verifies this (with respect to $\dot{c}$). Thus, let $G$ be $\P$-generic over $\V[H]$ with $p\in G$. Then in $\V[G]$, there is a $\sigma':L_{\bar{\mu}}[\bar{A}]\prec L_\mu[A]$ with $\sigma'``\bar{G}\sub G$ and $\sigma'(\dot{c})=\sigma_X(\dot{c})$. So $\sigma'$ lifts to $\sigma^*:L_{\bar{\mu}}[\bar{A}][\bar{G}]\prec L_\mu[A][G]$.

Since $\bar{H}\in L_{\bar{\mu}}[\bar{A}']$ and $\bar{G}$ is $\bar{\P}$-generic over
$L_{\bar{\mu}}[\bar{A}']$, it follows by the product lemma that $\bar{H}$ is $\bar{\Q}$-generic over $L_{\bar{\mu}}[\bar{A}][\bar{G}]$, and since $\sigma^*$ doesn't move $\bar{H}$, as $\bar{H}\sub\omega_1^{L_{\bar{\mu}}[\bar{A}]}=\omega_1^{L_{\bar{\mu}}[\bar{A}][\bar{G}]}$, it follows that $\sigma^*$ lifts to
\[\sigma^{**}:
L_{\bar{\mu}}[\bar{A}][\bar{G}][\bar{H}]\prec L_\mu[A][G][H]\]
Noting that
\[L_{\bar{\mu}}[\bar{A}][\bar{G}][\bar{H}]=L_{\bar{\mu}}[\bar{A}][\bar{H}][\bar{G}]=L_{\bar{\mu}}[\bar{A}'][\bar{G}]\] and
\[L_\mu[A][G][H]=L_\mu[A][H][G]=L_\mu[A'][G]\]
we see that $\sigma^{**}\rest L_{\bar{\mu}}[\bar{A}']$ witnesses that $X'$ elevates to $L_\mu[A']^\P$ in $\V[H]$, since $\sigma^{**}(c)=\sigma^*(\dot(c)^{\bar{H}})=\sigma'(\dot{c}^{\bar{H}})=\sigma'(\dot{c})^H=\sigma_X(\dot{c})^H=\sigma_{X'}(c)$. This last equality holds because $\sigma_X=\sigma_{X'}\rest L_{\bar{\mu}}[\bA]$, since $X$ is transitive in $X'$, in the sense that if $d\in X$ and $e\in d\cap X'$, then $e\in X$. It follows that $\sigma_{X'}$ is the lift of $\sigma_X$ to $L_{\bar{\mu}}[\bar{A}']$, and hence that $\sigma_X(\dot{c})^H=\sigma_{X'}(c)$.

Thus, $H_\theta^{\V[G]}\sub L_\mu[A']$ and $Z_{L_\mu[A'],\P,a}$ contains a club, because in $\V[H]$, $\{Y'\cap\omega_1\st Y'\prec L_\nu[A']\}$ contains a club.
\end{proof}

\begin{remark} Slight variations of the proof of the previous lemma show the following.
\begin{enumerate}[label=(\arabic*)]
\item Minimal subcompleteness is preserved by countably distributive proper (that is, strongly proper) forcing of size $\omega_1$. In a sense, the modified proof is somewhat easier than the original one.
\item The following slightly strengthened version of minimal subcompleteness of a forcing $\P$, which is still weaker than subcompleteness, is preserved by countably distributive forcing of size $\omega_1$: for any set $a$, there is a $\tau$ and an $A\sub\tau$, such that, letting $\mu=\tau^+$ and $\nu=\tau^{++}$, we have that $a,\P\in L_{\mu}[A]$, and for every countable $Y\prec L_\nu[A]$, $Y\cap L_\mu[A]$ elevates to $L_{\mu}[A]^\P$.
\end{enumerate}
\end{remark}

To formulate a corollary to the previous lemma, recall that, given a cardinal $\mu$, Jensen introduced a version of subcompleteness called \emph{subcompleteness above $\mu$,} which requires the elevated embedding to coincide with the originally given embedding up to the preimage of $\mu$. There is a natural version of minimal subcompleteness above $\mu$, which we make precise presently.

\begin{definition}
\label{def:MinimalSCabovemu}
Let $\mu$ be a cardinal, $N$ a transitive model of $\ZFC^-$, $\P$ a forcing notion and $X\prec N$ countable with $\mu,\P\in X$. Then $X$ \emph{elevates to $N^\P$ above $\mu$} if for every $\bar{G}$ which is generic over $N_X$ for $\sigma_X^{-1}(\P)$, and for every $c\in N_X$, there is a condition $p\in\P$ such that whenever $G$ is generic over $\V$ for $\P$, then in $\V[G]$, there is an elementary $\sigma':N_X\prec N$ with $(\sigma')``\bar{G}\sub G$, $\sigma'(c)=\sigma_X(c)$ and $\sigma'\rest\bar{\mu}=\sigma_X\rest\bar{\mu}$, where $\bar{\mu}=\sigma_X^{-1}(\mu)$.
\end{definition}

The proof of Lemma \ref{lemma:MinimalPartOfSubcompletenessIsPreserved} then shows the following.

\begin{corollary}
\label{lemma:MinimalPartOfSubcompletenessAbovemuIsPreserved}
Let $\P$ be subcomplete above $\mu$. Then after countably distributive forcing of size at most $\mu$, $\P$ is minimally subcomplete above $\mu$.
\end{corollary}

\section{Minimal subcompleteness and the preservation of properties of $\omega_1$-trees}
\label{sec:SCandPropertiesOfOmega1Trees}

Countably closed forcing does not add cofinal branches through $\omega_1$-trees, so it is natural to wonder whether other subcomplete forcing cannot do this either. Indeed, we see below that this is true of minimally subcomplete forcing as well. The proof for subcomplete forcing is given in \cite{Kaethe:Diss}.
Let's begin by establishing some terminology on trees.

\begin{definition}  A \emph{tree} is a partial order $T=\langle|T|, <_T \rangle$ in which the predecessors of any member of $|T$ are well-ordered by $<_T$ and there is a unique minimal element called the \emph{root}.
\begin{itemize}
	\item The members of $|T|$ are called the \emph{nodes} of $T$, and we will tend to conflate the tree $T$ with its underlying set $|T|$.
	\item The \emph{height of a node} $t \in T$ is the order type of the set of its predecessors under the restriction of the tree order. We write $T_\alpha$ for the $\alpha$th level of $T$, the set of nodes having height $\alpha$. The \emph{height of a tree} $T$, $\height(T)$, is the strict supremum of the heights of its nodes.
	\item We write $T\rest \alpha$ for the subtree of $T$ of nodes having height less than $\alpha$. An \emph{$\omega_1$-tree} is a normal tree of height $\omega_1$ where all levels are countable. A tree of height $\omega_1$ is \emph{normal} if every node has (at least) two immediate successors, nodes on limit levels are uniquely determined by their sets of predecessors, and every node has successors on all higher levels up to $\omega_1$.
	\item We write $T_t$ to denote the subtree of $T$ consisting of the nodes $s \in T$ with $s \geq_T t$. For nodes $t \in T$, by $\Succ_T(t)$ we mean the set of successors $s \geq_T t$ in the tree.
	\item A \emph{branch} $b$ in $T$ is a maximal linearly ordered subset of $T$, and the length of the branch is its order type. For $\alpha$ less than the length of $b$, we write $b(\alpha)$ for the node in $b$ that has height $\alpha$. We write $[T]$ for the set of \emph{cofinal branches} of $T$, that is, those branches containing nodes on every nonempty level of $T$. If $t \in T$ is a node, then we write $b_t$ to mean the ``branch" below $t$: $b_t = \set{ s \in T}{ s<_T t}$.
	\item An $\omega_1$-tree is an \emph{Aronszajn} tree if it has no cofinal branches. Two nodes $t$ and $s$ in $T$ are \emph{compatible}, written $s \parallel t$, if there is $r \in T$ such that $r \geq_T t$ and $r \geq_T s$. This is the same as demanding that either $s <_T t$, $s>_T t$, or $s=t$, or, in other words, that $s$ and $t$ are \textit{comparable}. Otherwise, they are incompatible, written $s \perp t$. An \emph{antichain} in a tree is a set of pairwise incompatible elements. A \emph{Suslin tree} is an $\omega_1$-tree with no uncountable antichain. When forcing with a tree, we reverse the order, so that stronger conditions are higher up in the tree. Consequently, Suslin trees are $ccc$ as notions of forcing. A \emph{Kurepa tree} is an $\omega_1$-tree with at least $\omega_2$-many cofinal branches.
\end{itemize}
\end{definition}

\begin{lemma}
\label{lemma:scbranch}
Let $T$ be an $\omega_1$-tree. If $\P$ is minimally subcomplete and $G$ is $\P$-generic then $[T]=[T]^{V[G]}$.
\end{lemma}

\begin{proof}
Assume not. Let $\dot b$ be a name for a new cofinal branch through $T \subseteq H_{\omega_1}$; let $p\in\P$ be a condition forcing that $\dot b$ is a new cofinal branch through $\check T$.
Let $N$ be a transitive model of $\ZFC^-$ with $\P,\dot{b},p\in N$ and $H_\theta\sub N$, such that $Z=Z_{N,\P,\kla{\dot{b},p}}$ contains a club, where $\theta$ is large enough to ensure that inside $N$, $p$ forces that $\dot b$ is a new cofinal branch through $\check T$. Let $\alpha\in Z$, and let $X$ witness this. Let $\alpha=\omega_1\cap X$, $\N=N_X$ and $\sigma=\sigma_X$. As usual, let $\bar{p},\bar{\P},\dot{\bar{b}}=\sigma^{-1}(p,\P,\dot{b})$.

By elementarity, we have that $\overline p$ forces $\overline{\dot b}$ to be a new cofinal branch over $\N$. As we construct a generic $\G$ for $\overline{\P}$ over $\N$, we will use the countability of $\N$ to diagonalize against all ``branches" as seen on level $\alpha$ of the tree $T$ in $N$, thereby obtaining a contradiction.


Toward this end, enumerate the dense sets $\seq{\overline D_n}{ n < \omega}$ of $\overline{\P}$ that belong to $\N$. Also denote the sequence of downward closures of nodes on level $\alpha$ of $T$, the ``branches" through $T\rest \alpha$ that extend to have nodes of higher height in $T$, as $\seq{b_n}{n < \omega}$. Now define a sequence of conditions of the form $\overline p_n$ for $n < \omega$ that decide values of $\overline{\dot b}$ in $\overline T$ differently from $b_n$. Ensure along the way that for all $n$,
\begin{itemize}
	\item $\overline p_{n+1} \in \overline D_n$ and
	\item $\overline p_{n+1} \leq \overline p_n$.
\end{itemize}
	
The construction (in $V$) may go as follows:

Let $\overline p_0 := \overline p \in \N$. For each $n < \omega$, note that there must be two conditions $\overline p^0_{n+1} \perp \overline p^1_{n+1}$, both extending $\overline p_n$, that decide the value of the branch $\overline{\dot b}$ to differ on some value. This always has to be possible since these conditions always extend $\overline p$, that forces $\overline{\dot b}$ to be new. Say $\overline p^1_{n+1} \forces \check x_{n} \in \overline{\dot b}$ and $\overline p^0_{n+1} \forces \check x_n \notin \overline{\dot b}$. Let $\overline p_{n+1}$ be a condition in $\overline D_n$ extending $\overline p^1_{n+1}$ if $x_n \notin b_n$, or a condition in $\overline D_n$ extending $\overline p^0_{n+1}$ otherwise.

Let $\G$ be the generic filter generated by the $\seq{ \overline p_n}{ n < \omega}$, let $\overline{\dot b}^{\G} = \overline b$. Since $\P$ is minimally subcomplete, there is a condition $q \in \P$ such that whenever $G$ is $\P$-generic with $q \in G$, by minimal subcompleteness we have $\sigma' \in V[G]$ such that: \begin{itemize}
	\item $\sigma': \N \prec N$
	\item $\sigma'(\overline \theta, \overline{\P}, \overline T, \overline p, \overline{\dot b}) = \theta, \P, T, p, \dot b$
	\item $\sigma'``\, \G \subseteq G$.
\end{itemize}
So below $q$ there is a lift $\sigma^*:\N[\G] \prec N[G]$ extending $\sigma'$ with $\sigma^*(\overline b) = \sigma'(\overline{\dot b})^G=\dot b^G = b$, and $\sigma^*(\overline T) = \sigma'(\overline T)^G=T$. The point is that since $\overline p \in \G$, we have $N[G] \models p \in G$, so $b$ is a branch through $T$.

Furthermore, $\alpha$ is the critical point of the embedding $\sigma^*$. So below $\alpha$ the tree $T$, and thus the branch $b$, is fixed. In particular, in $N[G]$, $b \rest \alpha = \overline b$. However, $\overline b$ was constructed so as to not be equal to any of the $b_n$s, so it cannot be extended to become a branch through $T$, since it can't have a node on the $\alpha$th level.
%
%
This is a contradiction.
\end{proof}

So in particular, minimally subcomplete forcing preserves Aronszajn trees. The following theorem shows that after forcing with a minimally subcomplete forcing, not only are there no new cofinal branches added to a Suslin tree $T$, but no uncountable antichains either. The proof is exactly the same as is given by Jensen \cite[Chapter 3 p.~10]{Jensen:FAandCH}.

\begin{lemma}
\label{lemma:Suslinpres}
Minimally subcomplete forcing preserves Suslin trees.
\end{lemma}

\begin{proof}
Let $T$ be a Suslin tree. Let $\P$ be minimally subcomplete. Suppose toward a contradiction that $p \in \P$ forces that $\dot A$ is a maximal antichain of size $\omega_1$.
Let $N$ be a transitive model of $\ZFC^-$ with $p,\P,\dot{A}\in N$, and with $H_\theta\sub N$, where $\theta$ is large enough that $N$ is sufficiently correct about what $p$ forces with respect to $\P$, such that $Z_{N,\P,\kla{p,\P,\dot{A}}}\neq\emptyset$. Let $X\prec N$ with $p,\P,\dot{A}\in X$ such that $X$ elevates to $N^\P$. Let $\sigma=\sigma_X$, $\N=N_X$.

Letting $\alpha = \omega_1^{\N}$, we have that $\overline T = T\rest \alpha$ as usual. Let $M$ be a countable, transitive $\ZFC^-$ model with both $\N, T\rest (\alpha+1) \in M$. Let $\G \subseteq \overline{\P}$ be generic over $M$ with $\overline p \in \G$. So $\G$ is also generic over $\N$.

We can now work below a condition in $G \subseteq \P$ generic to obtain a $\sigma' \in V[G]$ such that: \begin{itemize}
	\item $\sigma': \N \prec N$
	\item $\sigma'(\overline \theta, \overline{\P}, \overline T, \overline p, \overline{\dot A}) = \theta, \P, T, p, \dot A$
	\item $\sigma'``\, \G \subseteq G$.
\end{itemize}
As usual we have a lift $\sigma^*:\N[\G] \prec N[G]$. Letting $\overline A=\overline{\dot A}^{\G}$ and $\dot A^G=A$ we have that $\sigma^*(\overline A)=A$. Let $\seq{b_t}{t \in T_\alpha}$ be the collection of partial branches below the nodes of level $\alpha$ of the tree $T$.

Every node in $T$ above level $\alpha$ has to have a predecessor in level $\alpha$. For each $t \in T_\alpha$, $\G$ is $\overline{\P}$-generic over $\N[b_t]$ since $b_t$ is $\overline T$-generic over $\N$ - as cofinal branches through Suslin trees are generic. By the product lemma, each $b_t$ is $\overline T$-generic over $\N[\G]$. Since $\overline A$ is maximal, $b_t \cap \overline A \neq \emptyset$.
Thus $\overline A$ is sealed in $\overline T = T \rest \alpha$, meaning it has no elements above level $\alpha$. But since $\overline A \subseteq A$ and $A$ is maximal, this means that $A$ is countable, so $T$ remains Suslin as desired.
\end{proof}


The following rigidity properties were introduced in \cite{FuchsHamkins:DegreesOfRigidity}, where it was shown, among other things, that Suslin trees exhibiting these properties can be constructed, assuming the $\diamondsuit$ principle holds.

\begin{definition}
A normal $\omega_1$-tree $T$ has the \emph{unique branch property} (is $\ubp$) so long as $$\mathbbm 1 \forces_T `` \check T \text{ has exactly one new cofinal branch.}"$$ That is, after forcing with the tree, $T$ has exactly one cofinal branch that was not in the ground model. We say that $T$ has the \emph{$n$-fold $\ubp$} so long as forcing with $T^n$ adds exactly $n$ branches.

A Suslin tree is \emph{Suslin off the generic branch} so long as after forcing with $T$ to add a generic branch $b$, for any node $t$ not in $b$, the tree $T_t$ remains Suslin. Let $n$ be a natural number. A Suslin tree $T$ is \emph{$n$-fold Suslin off the generic branch} so long as after forcing with with the tree $n$ times, or forcing with $T^n$ that adds $n$ branches $b_1, \dots, b_n$, $T_p$ remains Suslin for any $p$ not on any $b_i$.
\end{definition}

Combining the results from Sections \ref{sec:FragmentsOfSubcompleteness} and \ref{sec:SCandPropertiesOfOmega1Trees}, we can conclude that these strong rigidity properties of Suslin trees are preserved by subcomplete forcing.

\begin{theorem} The following properties of an $\omega_1$-tree $T$ are preserved by subcomplete forcing:
\begin{enumerate}[label=(\arabic*)]
	\item \label{item:TisAronszajn} $T$ is Aronszajn
	\item \label{item:TisntKurepa} $T$ is not Kurepa
	\item \label{item:TisSuslin} $T$ is Suslin
	\item \label{item:TisSuslinandUBP} $T$ is Suslin and $\ubp$
	\item \label{item:TisSOB} $T$ is $\sob$
	\item \label{item:TisNSOB} $T$ is $n$-fold $\sob$ (for $n\ge 2$)
	\item \label{item:TisSuslinandNUBP} $T$ is $(n-1)$-fold Suslin off the generic branch and $n$-fold $\ubp$ (for $n\ge 2$)
\end{enumerate}
\end{theorem}
\let\qqed\qed
\begin{proof} \let\qed\relax
Items \ref{item:TisAronszajn}.~ and \ref{item:TisntKurepa}.~ are immediate corollaries of Lemma \ref{lemma:scbranch}. Item \ref{item:TisSuslin}.~ is Lemma \ref{lemma:Suslinpres}. In fact, these properties are even preserved by minimally subcomplete forcing.

For the remaining proofs, let $\P$ be a subcomplete forcing, and let $G$ be generic for $\P$ over $\V$.

\begin{proof}[Proof of \ref{item:TisSuslinandUBP}]
First we show upward absoluteness. Let $T$ be a Suslin tree with the $\ubp$. We have already seen that $T$ is still Suslin in $\V[G]$. To see that it is still $\ubp$, let $b$ be $T$-generic over $\V[G]$.
In $V[b]$, $b$ is the unique cofinal branch through $T$, and we have that $\P$ is still minimally subcomplete by Lemma \ref{lemma:MinimalPartOfSubcompletenessAbovemuIsPreserved}. Since minimally subcomplete forcing doesn't add branches to $\omega_1$-trees by Lemma \ref{lemma:scbranch}, $b$ is still the unique cofinal branch of $T$ in $V[b][G]=V[G][b]$. So $T$ still has the $\ubp$ in $V[G]$.

For downward absoluteness, suppose $T$ has the $\ubp$ in $V[G]$ but does not have the $\ubp$ in $V$. Let $p\in G$ force that $T$ has the $\ubp$. Let $b$ be a generic branch for $T$ over $\V$ such that in $V[b]$ the tree $T$ has at least two branches. Let $G'$ be $\P$-generic over $\V[b]$ with $p\in G'$. Then $T$ has at least two branches in $\V[b][G']=\V[G'][b]$, so $T$ is not $\ubp$ in $\V[G']$, contradicting that $p\in G'$.
\end{proof}

\begin{proof}[Proof of \ref{item:TisSOB}]
For upward absoluteness, let $T$ be \sob, and let $b$ be a generic branch for $T$ over $\V[G]$. $T$ is still Suslin in $V[G]$ by \ref{item:TisSuslin}. In $V[b]$, we have that $\P$ is still minimally subcomplete by Lemma \ref{lemma:MinimalPartOfSubcompletenessIsPreserved}. We have that for any node $t$ not in $b$, the tree $T_t$ remains Suslin in $V[b][G]=V[G][b]$ since after minimally subcomplete forcing $T_t$ remains Suslin by Lemma \ref{lemma:Suslinpres}. So $T$ remains \sob \ after forcing with $\P$.

For downward absoluteness, suppose $T$ is \sob \ in $V[G]$ but not in $V$. Let $p\in G$ force that $T$ is \sob. Let $b$ be a $\V$-generic branch through $T$ such that in $V[b]$ the tree $T$ is not \sob. So there is $t \in T$ off of $b$ such that $T_t$ is not Suslin in $\V[b]$. Let $G'$ be $\P$-generic over $\V[b]$ with $p\in G'$. Then, in $V[b][G']=\V[G'][b]$, $T_t$ is not Suslin, so that $T$ is not \sob{} in $\V[G']$, contradicting that $p\in G'$.
\end{proof}

\begin{proof}[Proof of \ref{item:TisNSOB}]
For upward absoluteness, let $T$ be $n$-fold \sob. Let $b_1\times b_2\times\ldots\times b_n$ be $T^n$-generic over $\V[G]$. Since $T$ is $n$-fold \sob, it follows that $T^n$ is countably distributive, so again we know that $\P$ is still minimally subcomplete in $\V[b_1,b_2,\ldots,b_n]$ by Lemma \ref{lemma:MinimalPartOfSubcompletenessIsPreserved}. In $V[b_1,\ldots,b_n][G]=V[G][b_1,\ldots,b_n]$, for any node $t$ not in one of the generic branches $b_1, \dots, b_n$, we have that $T_t$ is Suslin by Lemma \ref{lemma:Suslinpres}. So $T$ remains $n$-fold $\sob$ after forcing with $\P$. Downward absoluteness works as in \ref{item:TisSOB}.
\end{proof}

\begin{proof}[Proof of \ref{item:TisSuslinandNUBP}] \let\qed\qqed
For upward absoluteness, suppose that $T$ is $(n-1)$-fold Suslin off the generic branch and $n$-fold $\ubp$ for some $n\ge 1$.
Let $b_1\times b_2\times\ldots\times b_n$ be $T^n$-generic over $\V[G]$.
Since $T$ is $(n-1)$-fold \sob, it follows that $T^n$ is countably distributive, so again we know that $\P$ is still minimally subcomplete in $\V[b_1,b_2,\ldots,b_n]$ by Lemma \ref{lemma:MinimalPartOfSubcompletenessIsPreserved}.
Again we have $V[b_1,\ldots,b_n][G]=V[G][b_1,\ldots,b_n]$ where $b_1, \dots, b_n$ are the unique cofinal branches through $T$, since $\P$ does not add branches to $T^n$ over $V[b_1,\ldots,b_n]$. So $T$ has the $n$-fold $\ubp$ in $V[G]$. We have already seen in \ref{item:TisNSOB}.~that $T$ stays $(n-1)$-fold \sob{} in $\V[G]$.

Downward absoluteness is again the same as in \ref{item:TisSuslinandUBP}.
\end{proof}\end{proof}


\section{Generic absoluteness and the preservation of wide Aronszajn trees}
\label{sec:GenAbsWideAronszajn}

In Lemma \ref{lemma:scbranch}, we showed that subcomplete forcing cannot add a new branch to an $\omega_1$-tree, and in particular, that it preserves Aronszajn trees. In the present section, we will explore slightly stronger preservation properties.
Let us introduce versions of $\kappa$-trees in which the requirement that the levels have size less than $\kappa$ is relaxed.

\begin{definition}
\label{def:TreesWithWidths}
Let $\kappa$ and $\lambda$ be cardinals. We shall say that $T$ is a \emph{$(\kappa, {\leq}\lambda)$-tree} if $T$ is a tree of height $\kappa$ with levels of size less than or equal to $\lambda$. We shall refer to the size restriction on the levels in the tree in the second coordinate as the tree's width, so that a $(\kappa, {\leq}\lambda)$-tree has width $\le\lambda$.

An \emph{Aronszajn $(\kappa,{\leq}\lambda)$-tree} is a $(\kappa,{\leq}\lambda)$-tree with no cofinal branch. 
\end{definition}

It is easy to see that in general, countably closed forcing can't add a (cofinal) branch to any $(\omega_1, {\leq}\kappa)$-Aronszajn tree, for any $\kappa$. The guiding question for the work in the present section, as stated in \cite{Kaethe:Diss}, is as follows.

\begin{question}
\label{ques:aronszajnbranch?}
Can subcomplete forcing add cofinal branches to an $(\omega_1,{\leq}\omega_1)$-Aronszajn tree?
\end{question}

In the remainder of the present section, we will answer this question fully.
Let's first make the simple observation that even countably closed forcing may add branches to a tree of height $\omega_1$ and width ${\le}2^\omega$ (but such a tree can never be Aronszajn, by our earlier remarks).

\begin{observation}
Subcomplete (or even countably closed) forcing may add a cofinal branch to an $(\omega_1,{\leq}2^{\omega})$-tree.
\end{observation}

\begin{proof}
The point here is that the poset $\Add(\omega_1,1)$ is subcomplete since it is countably closed, but it may be viewed as a tree of height $\omega_1$ that has levels of size up to $2^{\omega}$. Of course this tree is not Aronszajn - it is Kurepa. Every cofinal branch through the tree corresponds to a subset of $\omega_1$, of which there are already more than $\omega_1$-many of in the ground model.
\end{proof}

If we allow the size of the levels of the height $\omega_1$ tree to be large, then we can obtain a slightly more complicated example of an Aronszajn tree to which subcomplete forcing can add a cofinal branch, using the forcing denoted by Jensen as $\P_A$. In the following definition, we write $\cof(\omega)$ for the class of ordinals with countable cofinality.
%

\begin{definition}
Let $\kappa > \omega_1$ be regular, and let $A \subseteq \kappa \cap \cof(\omega)$. We write \emph{$\P_{A}$} to denote the forcing designed to shoot a cofinal, normal sequence of order type $\omega_1$ through $A$. The conditions of $\P_A$ consist of normal functions of the form $p: \nu +1 \to A$, where $\nu< \omega_1$, and extension is defined in the usual way, by $p \leq q$ if and only if $q \subseteq p$.
\end{definition}

Jensen showed that $\P_A$ is subcomplete, see \cite{Jensen2014:SubcompleteAndLForcingSingapore}.
If $(\kappa \cap \cof(\omega))\setminus A$ is stationary in $\kappa \cap \cof(\omega)$, then $\P_A$ is not countably closed.
A $\P_A$-generic filter $G$ gives rise to the function $\cup G : \omega_1 \to A$ which is normal and cofinal in $\kappa$. The forcing $\P_A$ is used to show that the subcomplete forcing axiom $\SCFA$ implies \emph{Friedman's Principle}, which states that for every regular cardinal $\kappa > \omega_1$ and every stationary set $A \subseteq \kappa\cap\cof(\omega)$, there is a normal function $f: \omega_1 \to A$, that is, $A$ contains a closed set of order type $\omega_1$.

\begin{proposition}
\label{prop:PAaddsBranchToItself}
Suppose that Friedman's Principle fails for $\omega_2$. Then subcomplete forcing may add a cofinal branch to an $(\omega_1,{\leq}\omega_2 \cdot 2^\omega)$-Aronszajn tree. \end{proposition}

\begin{proof}
Let $A \subseteq \cof(\omega) \cap \omega_2$ witness the failure of Friedman's Principle. Consider the forcing poset $\P_A$ as a tree. It has size $\omega_1 \cdot \omega_2^\omega = \omega_2^\omega$, since it consists of functions with domain in $\omega_1$, and each condition is from a countable ordinal to $\omega_2$.

Considering $\P_A$ as a tree, it has height $\omega_1$. Each level has size less than or equal to $\omega_2^\omega = \omega_2 \cdot 2^\omega$. Moreover, since Friedman's Principle fails for $\omega_2$, the tree $\P_A$ has no cofinal branches and is thus Aronszajn, yet forcing with the tree will add a cofinal branch.
\end{proof}

However, we may slightly tweak the proof of Lemma \ref{lemma:scbranch} to see that
subcomplete forcing does preserve $(\omega_1,{<}2^\omega)$-Aronszajn trees. Note that this shows that the failure of \CH{} implies a negative answer to Question \ref{ques:aronszajnbranch?}. 

\begin{theorem}
\label{thm:SCforcingAddsNoNewBranchGeneral}
Subcomplete forcing cannot add (cofinal) branches to $(\omega_1,{<}2^\omega)$-trees. \end{theorem}

\begin{proof}
Assume not. Let $\dot b$ be a name for a new branch through $T \subseteq H_{\omega_1}$ an $(\omega_1, < 2^\omega)$-tree; let $p$ be a condition forcing that $\dot b$ is a new cofinal branch through $\check T$.
Let $\theta$ verify the subcompleteness of $\P$ and let's place ourselves in the standard setup: \begin{itemize}
	\item $\P \in H_\theta \subseteq N = L_\tau[A] \models \ZFC^-$ where $\tau>\theta$ and $A \subseteq \tau$
	\item $\sigma: \N \cong X \preccurlyeq N$ where $X$ is countable and $\N$ is full.
	\item $\sigma(\overline \theta, \overline{\P}, \overline T, \overline p, \overline{\dot b}) = \theta, \P, T, p, \dot b$.
\end{itemize}
Let $\alpha = \omega_1^{\N}$, the critical point of the embedding $\sigma$. By elementarity, we have that $\overline p$ forces $\overline{\dot b}$ to be a new branch over $\N$. We will construct continuum-many generics $\G_r$ for $\overline{\P}$ over $\N$, indexed by reals, each of which will correspond to continuum-many different values of the generic branch on level $\alpha$ of the tree $T$ in $N$, to obtain a contradiction.

Toward this end, enumerate the dense sets $\seq{D_n}{ n < \omega}$ of $\overline{\P}$ that belong to $\N$, so that $\overline p \in D_0$.

We would like to construct binary trees of conditions in $\overline \P$ and branches in $\overline T$: $ P= \seq{ \overline p_x }{ x \in 2^{<\omega} }$, $B = \seq{ \overline b_x }{ x \in 2^{< \omega} }$ such that, letting $|x|$ be the length of $x$, we have the following: \begin{itemize}
	\item $\overline p_x \in D_{|x|}$
	\item $x \subseteq y \implies \overline p_y \leq \overline p_x \leq \overline p$
	\item for some $\beta > |x|$, $p_x \forces \overline{\dot b} \rest \check \beta  = \overline b_x$
	\item $\overline b_{x^\frown \langle 0 \rangle} \neq \overline b_{x^\frown \langle 1 \rangle}$.
\end{itemize}

To do this, let $\overline p_0 := \overline p \in \N$. As we noted in the proof of Lemma \ref{lemma:scbranch}, there must be two conditions $\overline q^0_{1} \perp \overline q^1_{1}$, both extending $\overline p_0$, that decide the value of the branch $\overline{\dot b}$ differently. This always has to be true since these conditions extend $\overline p$, that forces that there is a new cofinal branch $\overline{\dot b}$. In particular, by our reasoning there is $\beta< \overline{\omega_1}$, and two conditions $\overline q_1^0 \leq \overline p_0$ and $\overline q_1^1 \leq \overline p_0$ such that $$\overline q_1^0 \forces \overline{\dot b}(\beta) = t^0 \text{ \ \ and \ \ } \overline q^1_{1} \forces \overline{\dot b}(\beta) = t^1,$$ where $t^0 \neq t^1$. Let $\overline b_{\langle 0 \rangle}$ be the branch in $\overline T$ below $t^0$ and $\overline b_{\langle 1 \rangle}$ be the branch in $\overline T$ below $t^1$. It must be that $\overline b_{\langle 0 \rangle} \neq \overline b_{\langle 1 \rangle}$. Moreover we may extend $\overline q^0_1$ and $\overline q^1_1$ to conditions $\overline p_{\langle 0 \rangle}, \overline p_{\langle 1 \rangle} \in \overline D_1$.

Continuing in such a fashion, we may recursively continue to define $\overline p_x$ for $x \in 2^{< \omega}$.

In particular, suppose $\overline p_x$ and $b_x$ are defined for $x$ of length $n$. By our reasoning above there is $\beta > |x|$, and two conditions above each $\overline p_x$, such that $\overline q_x^0 \leq \overline p_x$ and $\overline q_x^1 \leq \overline p_x$ so that
	$$\overline q_x^0 \forces \overline{\dot b}(\beta) = t^0 \text{\ \ and \ \ } \overline q^1_x \forces \overline{\dot b}(\beta) = t^1,$$
where $t^0 \neq t^1$. Let $\overline b_{x^\frown \langle 0 \rangle}$ be the branch in $\overline T$ below $t^0$ and $\overline b_{x^\frown \langle 1 \rangle}$ be the branch in $\overline T$ below $t^1$. Extending each of these incompatible conditions so as to land in the $(n+1)$th dense set, we find $\overline p_{x^\frown \langle 0 \rangle} \leq \overline q^0_x$ and $\overline p_{x^\frown \langle 1 \rangle} \leq \overline q^1_x$ such that $\overline p_{x^\frown \langle 0 \rangle}, \overline p_{x^\frown \langle 1 \rangle} \in D_{n+1}$.
Then, as desired, for each $x$ of length $n$ we have that $\overline p_{x^\frown \langle 0 \rangle}, \overline p_{x^\frown \langle 1 \rangle} \in D_{n+1}$. We've also designed it so that $\overline p_{x^\frown \langle 1 \rangle} \leq \overline p_x \leq \overline p$ and $\overline p_{x^\frown \langle 0 \rangle} \leq \overline p_x \leq \overline p$. Since we know that there is some $m'$ such that $\overline p_x \forces \overline{ \dot b} \rest m' = \overline b_x$ we know that $p_{x^\frown \langle 1 \rangle}$ and $p_{x^\frown \langle 0 \rangle}$ force the same thing since they both extend $\overline p_x$. Thus our construction gives us that $\overline p_{x^\frown \langle 0 \rangle} \forces \overline{ \dot b} \rest m = \overline b_{x^\frown \langle 0 \rangle}$ and $\overline p_{x^\frown \langle 1 \rangle} \forces \overline{ \dot b} \rest m' = \overline b_x$.

So we have our binary trees $P$ and $B$ as desired. Any chain of conditions in the binary tree $P$ will generate a generic filter $\overline G_r$; every real $r: \omega \to 2$ codes a path in the binary tree of conditions generating the generic. This is because our conditions were chosen to meet all of the dense sets in our list. Moreover, each generic filter $\overline G_r$ corresponds to a branch $\overline b_r$, where for each initial segment $t$ of $r$ satisfies that for some $\beta > |x|$, $\overline b_r \rest m = \overline b_x$. Because of how we chose $P$, this gives us that $$\N[\G] \models {\overline{\dot b}}^{\G} =\overline b_r.$$

For each $r$ let $\overline{\dot b}^{\G_r} = \overline b_r$. Since $\P$ is subcomplete, for each $r$ there is a condition $q_r \in \P$ such that whenever $G$ is $\P$-generic with $q_r \in G$, by subcompleteness we have $\sigma_r \in V[G]$ such that: \begin{itemize}
	\item $\sigma_r: \N \prec N$
	\item $\sigma_r(\overline \theta, \overline{\P}, \overline T, \overline p, \overline{\dot b}) = \theta, \P, T, p, \dot b$
	\item $\sigma_r``\, \G_r \subseteq G$.
\end{itemize}
So below each $q_r$ there is a lift $\sigma_r^*:\N[\G_r] \prec N[G]$ extending $\sigma_r$ with $\sigma_r^*(\overline b_r) = \sigma_r(\overline{\dot b})^G=\dot b^G = b_r$, and $\sigma_r^*(\overline T) = \sigma_r(\overline T)^G=T$.

This means that we may force over $N$ with $\P$ to obtain continuum many cofinal branches through the tree, $\seq{ b_r }{ r \in 2^\omega}$. Each of these branches must, of course, have a node on level $\alpha$. Since each cofinal branch $\overline b_r$ is unique, and since $\alpha$ is the critical point, this means that there are $2^\omega$-many nodes on level $\alpha$ of the tree $T$ in $N$, a contradiction.
\end{proof}

Note that Proposition \ref{prop:PAaddsBranchToItself} shows that if \CH{} fails, then the previous theorem is optimal, that is, it cannot be extended to $(\omega_1,{\le}2^\omega)$-trees.

Next, we will look at the preservation of wide Aronszajn trees from a different angle, and show that under \CH, dropping the restriction to trees with countable levels amounts to making a statement about a certain form of generic absoluteness which we introduce in the following.

\begin{definition}
\label{def:GenericAbsoluteness}
Let $n$ be a natural number, let $\P$ be a notion of forcing, and let $\kappa$ be a cardinal. Then \emph{$\P$-generic $\mathbf{\Sigma}^1_n(\kappa)$-absoluteness} is the statement that for any model $M=\kla{M,\vec{A}}$ of size $\kappa$ for a countable first order language and every $\Sigma^1_n$-sentence $\varphi$ over the language of $M$, the following holds:
\[(M \models \varphi)^V \iff 1_\P\forces_\P(M \models \varphi)\]
Note that we don't distinguish between the first and second order satisfaction symbol.


For a forcing class $\Gamma$, $\Gamma$-generic $\mathbf{\Sigma}^1_n(\kappa)$-absoluteness is the statement that $\P$-generic $\mathbf{\Sigma}^1_n(\kappa)$-absoluteness holds for every $\P\in\Gamma$. The classes of interest to us are the classes of ccc, proper, semi-proper, stationary set preserving or subcomplete forcings.
\end{definition}

We will mostly be interested in $\mathbf{\Sigma}^1_1(\kappa)$-absoluteness. Note that by upward absoluteness, $\mathbf{\Sigma}^1_1(\kappa)$-absoluteness for a forcing notion $\P$ can be equivalently expressed by saying that for $M$ as in the above definition and a $\Sigma^1_1$-formula $\varphi$, if $M\models\varphi$ holds in every forcing extension by $\P$, then $M\models\varphi$ holds in $\V$. It is a $\ZFC$ fact that countably-closed $\mathbf{\Sigma}_1^1(\omega_1)$-absoluteness holds, and more generally, ${<}\kappa$-closed $\mathbf{\Sigma}_1^1(\kappa)$-absoluteness holds, for regular $\kappa$, see \cite{Fuchs:MPclosed}. Much is known about the case $\kappa=\omega$. For example, by Shoenfield absoluteness, $\P$-generic $\mathbf{\Sigma}^1_2(\omega)$-generic absoluteness holds for any forcing notion $\P$. Here, we will mostly be interested in the case $\kappa=\omega_1$.
%
The following lemma gives an equivalent characterization under \CH{} for forcing notions which do not add reals, in terms of preserving wide Aronszajn trees.

\begin{lemma}
\label{lem:CharacterizationOfGenericAbsolutenessUnderCH}
Assume $\CH$. Let $\P$ be a forcing notion. Then the following are equivalent.
\begin{enumerate}[label=(\arabic*)]
\item
\label{item:NoNewRealsAndBranches}
Whenever $T$ is an $(\omega_1,{\le}\omega_1)$-Aronszajn tree, then
it is not the case that $1_\P$ forces that $\check{T}$ has a cofinal branch, and it is not the case that $1_\P\forces_\P$``there is a new real''.
\item
\label{item:Sigma1-1-Absoluteness}
$\P$-generic $\mathbf{\Sigma}_1^1(\omega_1)$-absoluteness holds.
\end{enumerate}
\end{lemma}

\begin{proof}
The direction \ref{item:Sigma1-1-Absoluteness}$\implies$\ref{item:NoNewRealsAndBranches} is clear: if $\mathbf{\Sigma}_1^1(\omega_1)$-statements are absolute for $\P$, then it cannot be that $1_\P$ forces that a real is added, because otherwise, by $\CH$, one could use a predicate $A\sub\omega_1$ which lists all reals, and the $\Sigma^1_1$-statement ``there is a subset $a$ of $\omega$ which is not listed in $A$'' would hold in any forcing extension by $\P$, but not in $\V$. Similarly, let $T$ be an $(\omega_1,{\le}\omega_1)$-Aronszajn tree. It cannot be that $1_\P$ forces that $\check{T}$ has a cofinal branch, because otherwise the $\Sigma^1_1$ statement ``$\check{T}$ has a cofinal branch'' would be true in $\V^\P$ but not in $\V$ - the nodes of $T$ can be assumed to be countable ordinals, and the tree ordering can be used as a binary predicate to express this.

Let's prove \ref{item:NoNewRealsAndBranches}$\implies$\ref{item:Sigma1-1-Absoluteness}. Since we are assuming $\CH$, we have that for any $(\omega_1,{\le}2^\omega)$-tree $T$, it is not the case that $1_\P$ forces that $\check{T}$ has a cofinal branch.
Upward absoluteness between $V$ and $V^\P$ clearly holds for $\Sigma^1_1$-statements.
To show downward absoluteness, let $\vec{A}$ be a finite list of finitary predicates on $\omega_1$, $\vec{A}\in V$.

Let $\psi(\vec{A})$ be the following statement:
\[\exists X \ (\omega_1,\vec{A}, X) \models \varphi\]
where $\varphi$ is a first order sentence in the language of set theory with predicate symbols for $\vec{A}$ and $X$. Assume that $\psi(\vec{A})$ is true in $V^\P$. Let $\dot{X}$ be a $\P$-name such that $1_\P$ forces that $\dot{X}$ is a witness that $\psi(\vec{\check{A}})$ holds.

In $\V$, let $T$ be the tree consisting of nodes of the form $(\alpha, x)$ such that $x \subseteq \alpha$, $\alpha<\omega_1$ and $(\alpha, \vec{A}\rest\alpha, x) \models \varphi(a)$, where
$\vec{A}\rest\alpha$ is the list whose elements are of the form $A_i\cap\alpha^{m_i}$, $m_i$ being the arity of $A_i$. The tree ordering $\leq$ is defined by setting
\[(\alpha, x) \leq (\beta, y) \iff (\alpha, \vec{A}\rest\alpha, x) \prec (\beta, \vec{A}\rest\beta, y).\]
Notice that $T$ has cardinality $\omega_1$ in $\V$, by \CH.

Now, if $G$ is an arbitrary filter $\P$-generic over $\V$, then
by a standard L\"{o}wenheim-Skolem style argument, applied in $\V[G]$, the set
\[C = \set{ \alpha < \omega_1 }{ (\alpha, \vec{A}\rest\alpha, \dot{X}^G \cap \alpha) \prec (\omega_1,\vec{A},\dot{X}^G) }\]
is club in $\omega_1$.	
Thus the set $\set{ (\alpha, \dot{X}^G \cap \alpha) }{ \alpha \in C }$ defines a cofinal branch through $T$ in $V[G]$, since for all countable $\alpha$, we have $\dot{X}^G \cap \alpha \in V$ as $\P$ doesn't add reals.

Since this works for any $G$, and since we assumed that for any $(\omega_1,{\le}\omega_1)$-Aronszajn tree, it is not the case that $1_\P$ forces that it has a cofinal branch, it follows that $T$ is not an $(\omega_1,{\le}\omega_1)$-Aronszajn tree. Hence, $T$ has a cofinal branch in $V$, call it $b$. Let
\[X' = \bigcup \set{ x }{ \exists \alpha < \omega_1 \ (\alpha, x) \in b }\]
Since $(\omega_1, \vec{A}, X')$ is the union of an elementary chain of models satisfying $\varphi$, this model must also satisfy $\varphi$ in $V$, and thus $\psi(\vec{A})$ holds in $V$ as witnessed by $X'$.
\end{proof}

The formulation of condition 1.~in the previous lemma seems a little cumbersome, and there is a clearer variant of the lemma, using a slightly modified version of $\mathbf{\Sigma}^1_1(\kappa)$-absoluteness, which we define presently.

\begin{definition}
\label{def:StrongAbsoluteness}
Let $\P$ be a poset and $\kappa$ a cardinal. Then \emph{strong $\P$-generic $\mathbf{\Sigma}^1_1(\kappa)$-absoluteness} is the principle asserting that
for any model $M=\langle M, \vec A\rangle$ of size $\kappa$ for a countable first order language and any $\Sigma^1_1$-sentence $\varphi$ over that language, whenever $G\sub\P$ is generic over $\V$, then $M\models\varphi$ iff $(M\models\varphi)^{\V[G]}$. Similarly, if $\Gamma$ is a forcing class, then \emph{strong $\Gamma$-generic $\mathbf{\Sigma}^1_1(\kappa)$-absoluteness} says that strong $\P$-generic $\mathbf{\Sigma}^1_1(\kappa)$-absoluteness holds for every $\P\in\Gamma$.

If $\P$ and $\Q$ are notions of forcing, then we say that $\P$ and $\Q$ are \emph{forcing equivalent} if they produce the same forcing extensions, that is, for every $\P$-generic $G$, there is a $\Q$-generic $H$ such that $\V[G]=\V[H]$ and vice versa.

Let us also introduce the notation $\P_{{\le}p}$ for the restriction of the ordering of $\P$ to the set of conditions $q\le p$. Call a forcing class $\Gamma$ \emph{natural} if whenever $\P\in\Gamma$ and $p\in\P$, then $\P_{{\le}p}$ is forcing equivalent to some $\Q\in\Gamma$.
\end{definition}

In other words, using $\Sigma^1_1$-upward absoluteness, for a model $M$ as above and a $\Sigma^1_1$ sentence over the language of $M$, \emph{strong} $\P$-generic $\mathbf{\Sigma}^1_1(\kappa)$-absoluteness says that for any $\P$-generic $G$, if $(M\models\varphi)^{\V[G]}$ holds, then $M\models\varphi$ holds. Regular $\P$-generic $\mathbf{\Sigma}^1_1(\kappa)$-absoluteness, on the other hand, says that if for every $\P$-generic $G$, $(M\models\varphi)^{\V[G]}$ holds, then $M\models\varphi$ holds. To clarify the difference, let's consider the class $\Gamma$ of all forcing notions $\P$ such that $\P$ is ccc and $\P$ has an atom. Then $\Gamma$-generic $\mathbf{\Sigma}^1_1(\kappa)$ always holds, because for $\P\in\Gamma$ and $M$, $\varphi$ as before, if $(M\models\varphi)^{\V[G]}$ holds for every $\P$-generic G, then it holds for some $G$ that contains an atom, in which case $\V[G]=\V$, and thus, $M\models\varphi$. On the other hand, strong $\Gamma$-generic $\mathbf{\Sigma}^1_1(\omega_1)$-absoluteness implies that every Aronszajn tree is special, because for an Aronszajn tree $T$, we can consider the lottery sum of a ccc forcing notion specializing $T$ and a trivial forcing, consisting of one atom. That forcing notion is in $\Gamma$. Let $G$ be generic for the nontrivial part of the forcing. If we let $M$ be an elementary submodel of $H_{\omega_1}$ of size $\omega_1$, equipped with $T$ as a predicate, then the existence of a function specializing $T$ can be expressed as a $\Sigma^1_1$ sentence over $M$, and it holds in $\V[G]$, hence in $\V$, which means that $T$ is special in $\V$.

It is easy to see that the notion of forcing equivalence introduced in the previous definition is first order expressible. Clearly, if $\P$ and $\Q$ are forcing equivalent, then (strong) $\P$-generic $\mathbf{\Sigma}^1_1(\kappa)$-absoluteness is equivalent to (strong) $\Q$-generic $\mathbf{\Sigma}^1_1(\kappa)$-absoluteness. The following is essentially a reformulation of \cite[Corollary 3.11]{FuchsRinot:WeakSquareStationaryReflection}.

\begin{fact}
\label{fact:SCforcingIsNatural}
The class of subcomplete forcing notions is natural.
\end{fact}

\begin{proof}
If $\P$ is subcomplete and $p\in\P$, then by \cite[Corollary 3.11]{FuchsRinot:WeakSquareStationaryReflection}, $\P_{{\le}p}$ is $\delta(\P)$-subcomplete (in the sense of \cite{Fuchs:ParametricSubcompleteness}). This means that $\P_{{\le}p}$ is essentially subcomplete in the sense of \cite[Definition 2.2]{Fuchs:ParametricSubcompleteness}, and \cite[Observation 2.4]{Fuchs:ParametricSubcompleteness} then implies that $\P_{{\le}p}$ is forcing equivalent to a subcomplete forcing.
\end{proof}

Let us make a simple observation relating strong absoluteness to the previously introduced version of absoluteness.

\begin{observation}
\label{observation:EquivalenceBtwStrongAndRegularAbsForNaturalForcingClasses}
Let $\kappa$ be a cardinal, let $\P$ be a forcing notion, and let $\Gamma$ be a forcing class.
\begin{enumerate}[label=(\arabic*)]
  \item
  \label{item:StrongIsForAllSpecializations}
  Strong $\P$-generic $\mathbf{\Sigma}^1_1(\kappa)$-absoluteness is equivalent to saying that $\{\P_{{\le}p}\st p\in\P\}$-generic $\mathbf{\Sigma}^1_1(\kappa)$-absoluteness holds.
  \item
  \label{item:StrongObsoleteForNaturalGamma}
  If $\Gamma$ is natural, then strong $\Gamma$-generic $\mathbf{\Sigma}^1_1(\kappa)$-absoluteness is equivalent to $\Gamma$-generic $\mathbf{\Sigma}^1_1(\kappa)$-absoluteness.
  \item
  \label{item:ClassesOfInterestAreNatural}
  If $\Gamma$ is either the class of all c.c.c., proper, semi-proper, countably closed, stationary set preserving or subcomplete forcing notions, then $\Gamma$ is natural, and hence strong $\Gamma$-generic $\mathbf{\Sigma}^1_1(\kappa)$-absoluteness is equivalent to $\Gamma$-generic $\mathbf{\Sigma}^1_1(\kappa)$-absoluteness.
\end{enumerate}
\end{observation}

\begin{proof}
For \ref{item:StrongIsForAllSpecializations}, assume that strong $\P$-generic $\mathbf{\Sigma}^1_1(\kappa)$-absoluteness holds, let $M$ be a $\kappa$-sized model of a countable first order language, let $\varphi$ be a $\Sigma^1_1$-sentence of in that language, and let $p\in\P$ be a condition. Assume that $1_{\P_{{\le}p}}$ forces (with respect to $\P_{{\le}p}$) that $M\models\varphi$. If $G\ni p$ is $\P$-generic, then in $\V[G]$, it is the case that $M\models\varphi$. Thus, by strong $\P$-generic $\mathbf{\Sigma}^1_1(\kappa)$-absoluteness, it is true in $\V$ that $M\models\varphi$. This shows that $\{\P_{{\le}p}\st p\in\P\}$-generic $\mathbf{\Sigma}^1_1(\kappa)$-absoluteness holds.

For the converse, assume that $\{\P_{{\le}p}\st p\in\P\}$-generic $\mathbf{\Sigma}^1_1(\kappa)$-absoluteness holds, let $M$ and $\varphi$ be as before, let $G$ be $\P$-generic over $\V$, and assume that in $\V[G]$, it is the case that $M\models\varphi$. There is then a condition $p\in G$ which forces that $M\models\varphi$. But then it follows that $1_{\P_{{\le}p}}$ forces that $M\models\varphi$. Hence, by
$\P_{{\le}p}$-generic $\mathbf{\Sigma}^1_1(\kappa)$-absoluteness, it follows that $M\models\varphi$ holds in $\V$.

Now \ref{item:StrongObsoleteForNaturalGamma} and \ref{item:ClassesOfInterestAreNatural} follow immediately from \ref{item:StrongIsForAllSpecializations}, using Fact \ref{fact:SCforcingIsNatural} and the remark preceding this fact.
\end{proof}

All of this could be done for $\Gamma$-generic $\mathbf{\Sigma}^1_n$-absoluteness as well, of course, but we will not need this generality here. We obtain the following version of the previous lemma.

\begin{lemma}
\label{lem:CharacterizationOfStrongGenericAbsolutenessUnderCH}
Assume $\CH$. Let $\P$ be a forcing notion. Then the following are equivalent.
\begin{enumerate}[label=(\arabic*)]
\item
\label{item:NoNewRealsAndBranches2}
$\P$ preserves $(\omega_1,{\le}\omega_1)$-Aronszajn trees and does not add reals.
\item
\label{item:Sigma1-1-Absoluteness2}
Strong $\P$-generic $\mathbf{\Sigma}_1^1(\omega_1)$-absoluteness holds.
\end{enumerate}
\end{lemma}

There is a natural version of the second order absoluteness properties introduced where one talks about a certain canonical structure $M$, defined by a formula to be re-interpreted in $\V[G]$. For example, let's define $\P$-generic $\mathbf{\Sigma}^1_1(H_{\omega_1})$-absoluteness to mean
\[(\langle H_{\omega_1},\vec{A}\rangle \models \varphi)^V \iff (\langle H_{\omega_1},\vec{A}\rangle \models \varphi)^{V^\P}\]
for any finite set of finitary predicates $\vec{A}$ and any $\Sigma^1_1$-sentence $\varphi$, where
$H_{\omega_1}$ is \emph{re-interpreted} in $\V^\P$ on the right hand side, rather than working with the same model on both sides. To be clear, the statement on the right hand side of the displayed equivalence means that $1_\P$ forces that $\langle H_{\omega_1},\vec{A}\rangle$, in the sense of the forcing extension, satisfies $\varphi$. Further, $\Gamma$-generic $\mathbf{\Sigma}^1_1(H_{\omega_1})$-absoluteness means that this holds for every $\P\in\Gamma$.

It turns out that $\Gamma$-generic $\mathbf{\Sigma}^1_1(H_{\omega_1})$-absoluteness
is equivalent to $\Gamma$-generic $\mathbf{\Sigma}^1_1(\kappa)$-absoluteness, where $\kappa=2^\omega$, if $\Gamma$ is a natural forcing class. Here, and in the following, we will indicate second order quantification by upper case variables.
To see the claimed equivalence, first suppose $\Gamma$-generic $\mathbf{\Sigma}^1_1(H_{\omega_1})$-absoluteness holds. It follows that $H_{\omega_1}=H_{\omega_1}^{\V[G]}$ whenever $G$ is generic for some $\P\in\Gamma$.
Clearly, $H_{\omega_1}\sub H_{\omega_1}^{\V^\P}$, for any $\P$.
But if we had $H_{\omega_1}^\V\subsetneqq M^{\V[G]}$, then we could take $A=H_{\omega_1}^\V$, and in $\V[G]$ it would be true that $\kla{H_{\omega_1}^{\V[G]},A}\models\exists x\quad\neg\dot{A}(x)$, but clearly this is not true in $\V$. So it follows that
$\Gamma$-generic $\mathbf{\Sigma}^1_1(\kappa)$-absoluteness also holds, because $H_{\omega_1}^\V$ has size $\kappa$ and doesn't change by forcing in $\Gamma$.

To see the converse, assume that $\Gamma$-generic $\mathbf{\Sigma}^1_1(\kappa)$-absoluteness holds. We claim that it follows that $H_{\omega_1}=H_{\omega_1}^{\V[G]}$ whenever $G$ is generic for some $\P\in\Gamma$. Suppose otherwise. Then $H_{\omega_1}\subsetneqq H_{\omega_1}^{\V[G]}$ for some $G$ generic for some $\P\in\Gamma$, which means that $\V[G]$ has a new real. But then, in $\V[G]$, the second order formula $\exists X\sub\omega\quad\forall x\quad x\neq X$, holds in the structure $\kla{H_{\omega_1}^\V,\in}$, while this is not true in $\V$. Thus, it follows that $\Gamma$-generic $\mathbf{\Sigma}^1_1(H_{\omega_1})$-absoluteness holds.

\begin{remark}
\label{remark:CHorNonCH}
In the Lemmas \ref{lem:CharacterizationOfGenericAbsolutenessUnderCH} and \ref{lem:CharacterizationOfStrongGenericAbsolutenessUnderCH}, $\mathbf{\Sigma}_1^1(\omega_1)$-absoluteness can be replaced with $\mathbf{\Sigma}_1^1(H_{\omega_1})$-absoluteness, since under $\CH$, $H_{\omega_1}$ has size $\omega_1$. On the other hand, if $\CH$ fails, then $\Add(\omega_1,1)$-generic $\mathbf{\Sigma}_1^1(H_{\omega_1})$-absoluteness fails, even though $\Add(\omega_1,1)$, being countably closed, preserves Aronszajn trees of any width. This is because $\Add(\omega_1,1)$ forces \CH, and this can be expressed as a $\mathbf{\Sigma}^1_1(H_{\omega_1})$ statement true in $\V^{\Add(\omega_1,1)}$ but false in $\V$.
\end{remark}


We will now explore a fruitful connection between subcomplete generic $\mathbf{\Sigma}^1_1(\omega_1)$-absoluteness and the bounded subcomplete forcing axiom.

The bounded forcing axiom was originally introduced in \cite{GoldsternShelah:BPFA} in the context of proper forcing. The bounded forcing axiom for a poset $\P$ says that if $\mathbb{B}$ is the complete Boolean algebra of $\P$, then for any collection of up to $\omega_1$ many maximal antichains in $\mathbb{B}$, each having size at most $\omega_1$, there is a filter in $\mathbb{B}$ that meets each antichain. The bounded forcing axiom for a class $\Gamma$ of forcings says that each $\P\in\Gamma$ satisfies the bounded forcing axiom for $\P$. We write \BSCFA{} for the bounded forcing axiom for the class of all subcomplete forcings.
The following is a version of a characterization of the bounded forcing axiom, due to Bagaria, tailored to the present context.

\begin{theorem}[{\cite[Theorem 5]{Bagaria:BFAasPrinciplesOfGenAbsoluteness}}]
\label{thm:bagaria}
Let $\P$ be a forcing notion. Then the following are equivalent:
\begin{enumerate}[label=(\arabic*)]
\item
\label{item:BFAforP}
The bounded forcing axiom holds for $\P$.
\item
\label{item:Sigma1Absoluteness}
$\P$-generic $\Sigma_1(H_{\omega_2})$-absoluteness holds, meaning:
if $\varphi(\vec{x})$ is a $\Sigma_1$-formula in the language of set theory and $\vec{a}\in H_{\omega_2}$, then
$\varphi(\vec{a})$ iff $\forces_\P\varphi(\vec{a})$.
\end{enumerate}
\end{theorem}

A very useful way of reformulating this theorem is as follows.

\begin{theorem}
\label{thm:bagaria2}
Let $\P$ be a forcing notion. Then the following are equivalent:
\begin{enumerate}[label=(\arabic*)]
\item
\label{item:BFAforP2}
The strong bounded forcing axiom holds for $\P$, meaning that the bounded forcing axiom holds for $\{\P_{{\le}p}\st p\in\P\}$.
\item
\label{item:Sigma1Absoluteness2}
Strong $\P$-generic $\Sigma_1(H_{\omega_2})$-absoluteness holds:
if $G$ is $\P$-generic over $\V$, then \[\kla{H_{\omega_2},\in}\prec_{\Sigma_1}\kla{H_{\omega_2},\in}^{\V[G]}.\]
\end{enumerate}
\end{theorem}

%

Of course, in the previous theorem, $H_{\omega_2}$ is reinterpreted in $\V[G]$ on the right hand side of \ref{item:Sigma1Absoluteness}.. Thus, \ref{item:Sigma1Absoluteness2}.~of Theorem \ref{thm:bagaria2} can be taken as a characterization of the bounded forcing axiom for a natural forcing class $\Gamma$.

We will show next that property \ref{item:Sigma1Absoluteness}.~is equivalent to $\P$-generic (or strong $\P$-generic in the case of Theorem \ref{thm:bagaria2}) $\mathbf{\Sigma}^1_1(\omega_1)$-absoluteness. For this, we will need an observation that is probably a folklore fact, but since it is important in the present context, we will provide a proof. We will work with the following natural way of coding elements of $H_{\omega_2}$.

\begin{definition}
A \emph{code} is a pair $\kla{R,\alpha}$, where $R\subset\omega_1\times\omega_1$, $\alpha<\omega_1$ and $\kla{\omega_1,R}$ is extensional and well-founded.

If $\kla{R,\alpha}$ is a code, then let $U_R$, $\sigma_R$ be the unique objects (given by Mostowski's isomorphism theorem) such that $U_R$ is transitive and $\sigma_R:\kla{U_R,{\in}\rest U_R}\To\kla{\omega_1,R}$ is an isomorphism. The \emph{set coded by $\kla{R,\alpha}$} is
\[c_{R,\alpha}=\sigma_R^{-1}(\alpha).\]
\end{definition}

Clearly, every member of $H_{\omega_2}$ has a code, and only members of $H_{\omega_2}$ have codes. Using codes, $\Sigma_1$ statements over $\kla{H_{\omega_2},\in}$ can be translated into $\Sigma^1_1$ statements over $\kla{H_{\omega_1},\in}$.

\begin{observation}
\label{obs:Translation}
Let $\varphi(v_0,\ldots,v_{n-1})$ be a $\Sigma_1$-formula. Then there is a $\Sigma^1_1$-formula $\varphi^c$ with free variables $X_0,x_0\ldots,X_{n-1},x_{n-1}$ (upper case variables being second order and lower case ones being first order) such that the following holds.
Let $a_0,\ldots,a_{n-1}\in H_{\omega_2}$, and let $\kla{R_0,\alpha_0},\ldots,\kla{R_{n-1},\alpha_{n-1}}$ be codes, such that $a_i$ is coded by $\kla{R_i,\alpha_i}$, for $i<n$. Then
\[\kla{H_{\omega_2},\in}\models\varphi(a_0,\ldots,a_{n-1}) \iff
\kla{L_{\omega_1},\in}\models\varphi^c(R_0,\alpha_0,\ldots,R_{n-1},\alpha_{n-1}).\]
\end{observation}

\begin{proof}
The construction of $\varphi^c$ proceeds by induction on $\varphi$. We will assume that $\varphi$ is presented in such a way that the only subformulas of $\varphi$ that are negated are atomic. Any formula can be written in this form.

If $\varphi$ is of the form $v_0=v_1$, then $\varphi^c(X_0,x_0,X_1,x_1)$ is defined in such a way that it expresses: 
there is an injective function $F:\omega_1\To\omega_1$ with $F(x_0)=x_1$, such that whenever $\beta_0 X_0 \beta_1\ldots X_0 \beta_m X_0 x_0$, then $F(\beta_0)X_1F(\beta_1)\ldots X_1F(\beta_m)R_1x_1$ and vice versa. Expressing the existence of such a function requires a second order existential quantification. Hence, the resulting formula $\varphi^c(X_0,x_0,X_1,x_1)$ can be written as a $\Sigma^1_1$ formula.

If $\varphi$ is of the form $v_0\in v_1$, then $\varphi^c(X_0,x_0,X_1,x_1)$ is defined to express: 
there is a $\beta<\omega_1$ such that $\beta X_1 x_1$, and such that the sentence of the form $(v_0=v_1)^c$ holds of $X_0,x_0,X_1,\beta$ (reducing to the previous case). The second order existential quantification occurring in $(v_0=v_1)^c$ can be pushed in front of the first order quantification (``there exists a $\beta<\omega_1$''), in this case simply because both are existential quantifications.

If $\varphi$ is of the form $\neg(v_0=v_1)$, then $\varphi^c(X_0,x_0,X_1,x_1)$ is defined to express: 
there are $U_0,U_1,F$ such that $U_0$ is closed under $X_0$-predecessors, $U_1$ is closed under $X_1$-predecessors and $F:\kla{U_0,X_0\cap U_0^2}\To\kla{U_1,X_1\cap U_1^2}$ is a maximal isomorphism, meaning that $F$ cannot be expanded beyond $U_0$, and it is not the case that $x_0\in U_0$, $x_1\in U_1$ and $F(x_0)=x_1$.

If $\varphi$ is of the form $\neg(v_0\in v_1)$, then this can be expressed equivalently by $\forall v\in v_1 \neg (v_0=v)$. We already know how to translate $\neg(v_0=v)$, and we can then use the definition in the case of bounded quantification below.


The inductive steps corresponding to the logical connectives $\land$ and $\lor$ can be dealt with in the obvious way, setting $(\varphi\land\psi)^c=\varphi^c\land\psi^c$ and $(\varphi\lor\psi)^c=\varphi^c\lor\psi^c$.

Let's look at the case that $\varphi$ is of the form $\forall u\in w\quad\psi(u,w,v_0,\ldots,v_{n-1})$. Define the formula $\varphi^c(Y,y,X_0,x_0,\ldots,X_{n-1},x_{n-1})$ to express: 
for all $\beta Y y$, $\psi(u,w,v_0,\ldots,v_{n-1})^c$ is true of $Y,\beta,Y,y,X_0,x_0,\ldots,X_{n-1},x_{n-1}$. The resulting formula has a universal first order quantification over a $\Sigma^1_1$ formula. Since $\omega_1$-sequences of subsets of $\omega_1$ can be coded by single subsets of $\omega_1$, the second order quantification can be pulled out in front of the first order quantifier, resulting in a $\Sigma^1_1$ formula.

The case of existential bounded quantification is easier, so we omit it here.

Thus, we have described how to translate $\Sigma_0$-formulas. The remaining case is that $\varphi$ is of the form $\exists u\quad\psi(u,v_0,\ldots,v_{n-1})$, where $\psi(u,v_0,\ldots,v_{n-1})$ is a $\Sigma_0$-formula. In this case, the translated formula $\varphi^c(X_0,x_0,\ldots,X_{n-1},x_{n-1})$ expresses that
there are an $S$ (this is second order) and an $\alpha$ such that $\kla{S,\alpha}$ is a code and such that $\psi^c(S,\alpha,X_0,x_0\ldots,X_{n-1},x_{n-1})$ holds. 
Expressing that $\kla{S,\alpha}$ is a code amounts to saying that it is extensional, which is first order expressible, and that it is well-founded.
In order to do this, we use an additional existential second order quantification, saying that there is an $F\sub\omega_1\times\omega_1\times\omega_1$ such that, if we set $f_\xi=\{\kla{\gamma,\delta}\st\kla{\xi,\gamma,\delta}\in F\}$, then $f_\xi:\kla{\xi,S\cap(\xi\times\xi)}\To\kla{\omega_1,<}$ is an order preserving function, for every $\xi<\omega_1$. This can be expressed in a first order way, using the predicates $S$ and $F$ inside $L_{\omega_1}$, and it follows that $S$ is well-founded, because any decreasing $\omega$-sequence in $S$ would be bounded by some $\xi<\omega_1$, contradicting that $f_\xi$ is order preserving. And if $S$ is well-founded, then so is every initial segment $\kla{\xi,S\cap(\xi\times\xi)}$, hence there is an $f_\xi$ as described.
\end{proof}

Note that the proof of the previous observation contained a concrete translation procedure $\varphi\mapsto\varphi^c$ which \ZFC-provably has the properties described, that is, the same translation procedure works in any \ZFC-model. Note also that we could have used any other model of (a fairly weak fragment of) \ZFCm that contains $\omega_1$ in place of $L_{\omega_1}$. We will use this uniformity of the translation procedure in the following proof.

\begin{observation}
\label{obs:EquivalenceBtwSigma1-1AndSigma1-absoluteness}
Let $\P$ be a notion of forcing that preserves $\omega_1$.  Then the following are equivalent:
\begin{enumerate}[label=(\arabic*)]
\item
\label{item:Sigma1AbsolutenessAgain}
$\P$-generic $\Sigma_1(H_{\omega_2})$-absoluteness holds.
\item
\label{item:Sigma1-1AbsolutenessAgain}
$\P$-generic $\mathbf{\Sigma}^1_1(\omega_1)$-absoluteness holds.
\end{enumerate}
\end{observation}

\begin{proof}
The implication \ref{item:Sigma1AbsolutenessAgain}$\implies$\ref{item:Sigma1-1AbsolutenessAgain} is easy to see, because $\power(\omega_1)\sub H_{\omega_2}$, so a second order existential quantification over $\omega_1$ can be expressed as a first order existential quantification over the elements of $H_{\omega_2}$ which are subsets of $\omega_1$.

For the direction \ref{item:Sigma1-1AbsolutenessAgain}$\implies$\ref{item:Sigma1AbsolutenessAgain}, let $\vec{a}=a_0,\ldots,a_{n-1}$ be a list of parameters in $H_{\omega_2}$, $\varphi(\vec{x})$ a $\Sigma_1$-formula, and suppose that
\[\kla{H_{\omega_2}^{\V[G]},\in}\models\varphi(\vec{a})\]
whenever $G$ is $\P$-generic over $\V$.
In $\V$, let
$\kla{R_0,\alpha_0},\ldots,\kla{R_{n-1},\alpha_{n-1}}$ be codes for $a_0,\ldots,a_{n-1}$, respectively, and let $\varphi^c$ be the $\Sigma^1_1$ translation of $\varphi$ given by Observation \ref{obs:Translation}. Since the same codes work in $\V[G]$, the translation procedure is uniform, we can conclude that
\[(L_{\omega_1}\models\varphi^c(R_0,\alpha_0,\ldots,R_{n-1},\alpha_{n-1}))^{\V[G]}.\]
Clearly, $\varphi^c$ can be replaced by a $\Sigma^1_1$-sentence $\tilde{\varphi}^c$ in the language with predicate/constant symbols $\dot{C}_0,\dot{c}_0,\ldots,\dot{C}_{n-1},\dot{c}_{n-1}$ for the codes, so that we get
\[(\kla{L_{\omega_1},\in,R_0,\alpha_0\ldots,R_{n-1},\alpha_{n-1}}\models\tilde{\varphi}^c)^{\V[G]}\]
Since this model is in $\V$, and it has size $\omega_1$ there, noting that this holds for every $\P$-generic $G$, it follows from $\mathbf{\Sigma}^1_1(\omega_1)$-absoluteness that
\[\kla{L_{\omega_1},\in,R_0,\alpha_0,\ldots,R_{n-1},\alpha_{n-1}}\models\tilde{\varphi}^c\]
holds in $\V$, that is,
\[\kla{L_{\omega_1},\in}\models\varphi^c(R_0,\alpha_0,\ldots,R_{n-1},\alpha_{n-1})\]
holds in $\V$, which means that, undoing the translation, which is uniform, we get that
\[\kla{H_{\omega_2},\in}\models\varphi(\vec{a})\]
as desired.
\end{proof}

The same proof shows the equivalence of the strong forms of these generic absoluteness conditions. Note that either condition \ref{item:StrongSigma1AbsolutenessAgain} or \ref{item:StrongSigma1-1AbsolutenessAgain} of the following observation  implies that $\P$ preserves $\omega_1$.

\begin{observation}
\label{obs:EquivalenceBtwStrongSigma1-1AndStrongSigma1-absoluteness}
Let $\P$ be a notion of forcing. Then the following are equivalent:
\begin{enumerate}[label=(\arabic*)]
\item
\label{item:StrongSigma1AbsolutenessAgain}
Whenever $G$ is generic for $\P$ over $\V$, we have that
\[\kla{H_{\omega_2},\in}\prec_{\Sigma_1}\kla{H_{\omega_2},\in}^{\V[G]}.\]
\item
\label{item:StrongSigma1-1AbsolutenessAgain}
Strong $\P$-generic $\mathbf{\Sigma}^1_1(\omega_1)$-absoluteness holds.
\end{enumerate}
\end{observation}

Note that either 1.~or 2.~in the previous observation imply that $\P$ preserves $\omega_1$.

In general, we have the following simple observation.

\begin{observation}
\label{observation:1=>2=>3}
Let $\Gamma$ be a natural class of forcing notions, and consider the following statements.
\begin{enumerate}[label=(\arabic*)]
\item 
\label{item:BFA_Gamma}
$\BFA_\Gamma$.
\item 
\label{item:Gamma-genericSigma1-1-absoluteness}
$\Gamma$-generic $\mathbf{\Sigma}^1_1(\omega_1)$-absoluteness.
\item 
\label{GammaPreservesOmega1Omega1AroszajnTrees}
Forcings in $\Gamma$ preserve $(\omega_1,{\le}\omega_1)$-Aronszajn trees.
\end{enumerate}
Then (1)$\iff$(2)$\implies$(3).
\end{observation}

\begin{proof}
\ref{item:BFA_Gamma}$\iff$\ref{item:Gamma-genericSigma1-1-absoluteness} follows from Observations \ref{observation:EquivalenceBtwStrongAndRegularAbsForNaturalForcingClasses}.\ref{item:StrongObsoleteForNaturalGamma}, \ref{obs:EquivalenceBtwStrongSigma1-1AndStrongSigma1-absoluteness} and
Theorem \ref{thm:bagaria2}. The implication \ref{item:Gamma-genericSigma1-1-absoluteness}$\implies$\ref{GammaPreservesOmega1Omega1AroszajnTrees} follows because if $T$ were an $(\omega_1,{\le}\omega_1)$-Aronszajn tree that acquires a cofinal branch in $\V[G]$, where $G$ is $\P$-generic for some $\P\in\Gamma$, then the existence of such a branch would be a $\Sigma^1_1(T)$ statement true in $\V[G]$ but false in $\V$, contradicting \ref{item:Gamma-genericSigma1-1-absoluteness}.
\end{proof}

So, writing \BSCFA{} for the bounded subcomplete forcing axiom, we get the following equivalences, using the fact that the class of subcomplete forcing notions is natural (see Observation \ref{observation:EquivalenceBtwStrongAndRegularAbsForNaturalForcingClasses}), as well as
Theorem \ref{thm:bagaria2}, Lemma
\ref{lem:CharacterizationOfStrongGenericAbsolutenessUnderCH} and Observation \ref{obs:EquivalenceBtwSigma1-1AndSigma1-absoluteness}.

\begin{theorem}
\label{thm:UnderCHEverythingIsClear}
Assuming \CH, the following are equivalent.
\begin{enumerate}[label=(\arabic*)]
\item \BSCFA.
\item Subcomplete generic $\mathbf{\Sigma}^1_1(\omega_1)$-absoluteness.
\item Subcomplete forcing preserves $(\omega_1,{\le}\omega_1)$-Aronszajn trees.
\end{enumerate}
Actually, (1) and (2) are equivalent, regardless of whether \CH holds or not, and (1)/(2) always implies (3), but for the converse, we need \CH, since the failure of $\CH$ implies (3) (by Theorem \ref{thm:SCforcingAddsNoNewBranchGeneral}), but not (1)/(2).
\end{theorem}

Obviously, this theorem generalizes to any natural class $\Gamma$ of forcing notions that don't add reals.

Let us make some remarks on the consistency strength of \BSCFA{} and its relationship to \CH. It was shown in \cite{Fuchs:HierarchiesOfForcingAxioms} that the consistency strength of \BSCFA{} is a reflecting cardinal. Moreover, looking at the construction there, one sees that the consistency strengths of \BSCFA{} and of $\BSCFA+\CH$ are the same. Namely, if $\BSCFA$ holds, then $\kappa=\omega_2$ is reflecting in $L$, and one can perform a subcomplete (in the sense of $L$) forcing over $L$ to reach a model $L[g]$ where $\omega_2=\kappa$. Since the forcing is subcomplete in $L$, it does not add reals, and hence preserves \CH.
In a sense, in the context of \BSCFA, it is natural to assume \CH, since it holds in the ``natural'' models, and since \CH{} is implied by natural strengthenings of \BSCFA, such as the resurrection axiom or the maximality principle for subcomplete forcing, see \cite{Kaethe:Diss}, \cite{Fuchs:HierarchiesOfRA}.

However, \BSCFA{} does not imply \CH, and in fact, the first author, in joint work with Corey Switzer, observed that the consistency strength of $\BSCFA+\neg\CH$ is equal to that of $\BSCFA$, that is, the existence of a reflecting cardinal. Thus, assuming $\neg\CH$, condition 3.~in the previous theorem holds, while the consistency strength of the equivalent conditions 1.~and 2.~is a reflecting cardinal, showing that the implication cannot be reversed.

This puts us in a position to answer Question \ref{ques:aronszajnbranch?}, asking whether subcomplete forcing may add a cofinal branch to an $(\omega_1,{\le}\omega_1)$-Aronszajn tree, completely. Recall Theorem \ref{thm:SCforcingAddsNoNewBranchGeneral}, which gives us part 1.~of the following theorem.

\begin{theorem}
\label{theorem:CompleteAnswer}
Splitting in two cases, we have:
\begin{enumerate}[label=(\arabic*)]
  \item If \CH{} fails, then subcomplete forcing preserves $(\omega_1,{\le}\omega_1)$-Aronszajn trees.
  \item If \CH{} holds, then subcomplete forcing preserves $(\omega_1,{\le}\omega_1)$-Aronszajn trees iff \BSCFA{} holds.
\end{enumerate}
\end{theorem}

It is now interesting to explore the relationships between bounded forcing axioms, the forms of generic $\Sigma^1_1$-absoluteness introduced above, and the property of $(\omega_1,{\le}\omega_1)$-Aronszajn tree preservation for other canonical classes of forcing. Let's first observe some limitations on $\Gamma$-generic $\mathbf{\Sigma}^1_1(\kappa)$-absoluteness.

\begin{observation}
\label{obs:H-omega1NotInteresting}
\begin{enumerate}[label=(\arabic*)]
\item
\label{item:AddOmega11}
If $\CH$ fails, then $\Add(\omega_1,1)$-generic $\mathbf{\Sigma}^1_1(\omega_2)$-absoluteness fails.
\item
\label{item:Collapse}
$\Col{\omega_1,\omega_2}$-generic $\mathbf{\Sigma}^1_1(\omega_2)$-absoluteness fails.
\item
\label{item:AddingReals}
If $\P$ is a forcing that adds a real, then $\P$-generic $\mathbf{\Sigma}^1_1(2^\omega)$-absoluteness fails.
\end{enumerate}
\end{observation}

\begin{proof}
For \ref{item:AddOmega11}, if \CH{} fails, we can take $M$ to be an elementary submodel of $H_{\omega_1}$ of size $\omega_2$, containing $\omega_2$ many distinct reals. If $G$ is generic for $\Add(\omega_1,1)$ over $\V$, then in $\V[G]$, $M$ satisfies the $\Sigma^1_1$ statement that there is a function $F$ with domain $\omega_1$ that lists all the reals of $M$, but it is false in $\V$.

For \ref{item:Collapse}, we can take $M$ to be the structure $\omega_2$, equipped with an ordinal pairing function and a constant symbol for $\omega_1$. Then if $G$ is $\Col{\omega_1,\omega_2}$-generic over $\V$, in $\V[G]$, the structure $M$ satisfies the $\Sigma^1_1$ sentence expressing that there is a surjection $F$ from $\omega_1$ onto the universe of $M$. This is not true in $\V$.

For \ref{item:AddingReals}, if $\P$ adds a real, then the $\Sigma^1_1$-formula ``there is an $X\sub\omega$ such that for all $x$, $x\neq X$'' holds, from the point of view of $\V^\P$, in the structure $\kla{H_{\omega_1}^\V,\in}$, but not from the point of view of $\V$.
\end{proof}

So, by \ref{item:Collapse} of the above observation, if $\Gamma$ is the class of proper, semi-proper or stationary set preserving, subcomplete, or countably closed forcings, then $\Gamma$-generic $\mathbf{\Sigma}^1_1(\omega_2)$-absoluteness fails.

\begin{theorem}
\label{thm:GeneralSituation}
Let $\Gamma$ be the class of proper, semi-proper, stationary set preserving, ccc or subcomplete forcing notions. Consider the following properties. \begin{enumerate}[label=(\arabic*)]
\item $\BFA_\Gamma.$
\item $\Gamma$-generic $\mathbf{\Sigma}^1_1(\omega_1)$-absoluteness.
\item Forcings in $\Gamma$ preserve $(\omega_1,{\le}\omega_1)$-Aronszajn trees.
\end{enumerate}
Then (1)$\iff$(2)$\implies$(3), but (3) does not imply (1)/(2).
\end{theorem}

\begin{proof}
By Observation \ref{observation:1=>2=>3}, we know that (1)$\iff$(2)$\implies$(3) holds. Let's show that (3) does not imply (2).

In the case of subcomplete forcing, we have already seen that (3) follows from $\neg\CH$, while (1)/(2) have consistency strength a reflecting cardinal.

For the case of ccc forcing, recall that is known that $\CH$ is consistent with the statement that every Aronszajn tree is special, see \cite{MildenbergerShelah:SpecialisingAronszajnTrees}. But if every Aronszajn tree is special, then ccc forcing cannot add a cofinal branch to any $(\omega_1,{\le}\kappa)$-Aronszajn tree $T$, no matter how wide it is: assume $\P$ were a c.c.c.~forcing that did. Let $\dot{b}$ be a $\P$-name for a cofinal branch through $T$, and let $p\in\P$ force this. Let $X$ be the set of members $x$ of $T$ such that some $q\le p$ forces that $\check{x}\in\dot{b}$. Then $X$ is closed under $T$-predecessors, because if $x\le y\in X$ and $q\le p$ forces that $\check{y}\in\dot{b}$, then $q$ also forces that $\check{x}\in\dot{b}$. Also, the set $X$ has nodes at arbitrarily large heights less than $\omega_1$, since $p$ forces that $\dot{b}$ is a cofinal branch. Moreover, for any $\alpha<\omega_1$, $X$ has at most countably many nodes at level $\alpha$ of $T$, because for any such node, there is a condition below $p$ that forces that that node is in $\dot{b}$, and these conditions have to be pairwise incompatible, so that the claim follows from the fact that $\P$ is c.c.c.
This shows that the restriction $\bar{T}$ of $T$ to $X$ is an $\omega_1$-tree, hence an Aronszajn tree, and hence special. Now we have a contradiction, since $\P_{{\le}p}$ adds a branch to $\bar{T}$. This is impossible, since $\P$ preserves $\omega_1$. But now, in any model of \CH{} in which every Aronszajn tree is special, (3) is satisfied, while (1) and (2) fail, since (1)/(2) imply the failure of \CH.

To cover the remaining cases, we will show that if $\Gamma$ is a natural forcing class containing all proper forcing notions, then the assertion that forcing notions in $\Gamma$ preserve $(\omega_1,{\le}\omega_1)$-Aronszajn trees does not imply $\BFA_\Gamma$.
To see this, recall that it follows from $\MA_{\omega_1}$ that every $(\omega_1,{\le}\omega_1)$-Aronszajn tree is special, and hence that every such tree is preserved by every $\omega_1$-preserving forcing. But the consistency strength of $\MA_{\omega_1}$ is the same as that of $\ZFC$, while the consistency strength of $\BFA_\Gamma$ is at least a reflecting cardinal. \end{proof}

Recall that under $\CH$, the versions of the three conditions listed in the previous theorem for subcomplete forcing are equivalent. The proof showed that this is not the case for c.c.c.~forcing (for (3) is consistent with \CH, in this case, while (1) is not). Of course, subcomplete forcing is the only class considered here whose bounded forcing axiom is consistent with \CH, and it is in the context of \CH{} that we have this unusual equivalence between the three conditions.

%
%

%

\bibliographystyle{abbrv}
\bibliography{literatur}

\end{document}